\documentclass[a4paper]{amsart}
\usepackage{color}
\usepackage{microtype}	
\usepackage{booktabs}

\usepackage{hyperref}
\usepackage[capitalise]{cleveref}	
\usepackage{enumitem}	
\usepackage{amssymb}
\usepackage{mathtools}	
\mathtoolsset{centercolon}	
\usepackage{tensor}		
\usepackage{fancyhdr}

\usepackage{multirow}

\usepackage[url=false]{biblatex}
\AtEveryBibitem{
	\ifentrytype{online}{%
	\clearfield{url}%
	\clearfield{urlyear}%
}{}}
\newtheorem{theorem}{Theorem}

\newtheorem{proposition}{Proposition}[section]
\newtheorem{lemma}[proposition]{Lemma}
\newtheorem{corollary}[proposition]{Corollary}

\theoremstyle{definition}

\newtheorem{definition}[proposition]{Definition}
\newtheorem{example}[proposition]{Example}

\theoremstyle{remark}
\newtheorem{remark}[proposition]{Remark}

\DeclarePairedDelimiter{\norm}{\lVert}{\rVert}
\DeclarePairedDelimiter{\abs}{\lvert}{\rvert}
\DeclarePairedDelimiterXPP{\evalAt}[2]{}{.}{\rvert}{_{#1}}{#2}

\newcommand{\Liederivative}[1]{\mathcal{L}_{#1}}
\newcommand{\contractwith}[1]{\iota_{#1}}
\newcommand{\dif}{\mathop{}\!d}
\DeclareMathOperator{\divergence}{div}

\DeclareMathOperator{\trace}{tr}
\DeclareMathOperator{\HodgeStar}{\star}
\newcommand{\tensorproduct}{\otimes}
\newcommand{\identityOperator}{1}
\newcommand{\Aut}{\mathrm{Aut}}
\newcommand{\actsOn}{\curvearrowright}

\newcommand{\odfrac}[2]{\frac{\dif #1}{\dif #2}}

\newcommand{\Liegroup}[1]{\mathrm{#1}}
\newcommand{\Liealgebra}[1]{\mathfrak{#1}}

\newcommand{\Reebfield}{R}		
\newcommand{\emptyslot}{\cdot}	
\newcommand{\difalphasharp}{\dif \alpha^+}
\newcommand{\dtFactor}{\tau}	
\newcommand{\modelLieGroup}{\Liegroup{Sol}}
\newcommand{\modelLieAlgebra}{\Liealgebra{sol}}

\title{Cosymplectic Chern--Hamilton conjecture}
\author[S. Dyhr]{Søren Dyhr}
\address{
	Søren Dyhr, Laboratory of Geometry and Dynamical Systems, Universitat Po\-li\-tèc\-ni\-ca de Ca\-ta\-lunya, Avinguda del Doctor Marañon 44-50, Barcelona 08028, Spain
	and
	Centre de Recerca Matemàtica, Edifici C, Campus Bellaterra, 08193 Bellaterra, Spain.
}
\email{soren.dyhr@upc.edu}

\author[Á. González-Prieto]{Ángel González-Prieto}
\address{ Ángel González-Prieto, Departamento de Álgebra, Geometría y Topología, Facultad de CC.\ Matemáticas, Universidad Complutense de Madrid, 28040 Madrid, Spain and Instituto de Ciencias Matemáticas (CSIC-UAM-UCM-UC3M),
28049 Madrid, Spain.}
\email{angelgonzalezprieto@ucm.es}

\author[E. Miranda]{Eva Miranda}
\address{Eva Miranda, Laboratory of Geometry and Dynamical Systems and SYMCREA, Department of Mathematics, EPSEB, Universitat Politècnica de Catalunya-IMTech in Barcelona and
\\
CRM Centre de Recerca Matemàtica, Campus de Bellaterra Edifici C, 08193 Bellaterra, Barcelona.}
\email{eva.miranda@upc.edu}

\author[D. Peralta-Salas]{Daniel Peralta-Salas}
\address{Daniel Peralta-Salas, Instituto de Ciencias Matemáticas, Consejo Superior de Investigaciones Científicas, 28049 Madrid, Spain.}
\email{dperalta@icmat.es}

\thanks{S.\ Dyhr received the support of a fellowship from the ”la Caixa” Foundation (ID 100010434) with fellowship code LCF/BQ/DI23/11990085. Á.\ González-Prieto is supported by the Spanish State Research Agency project PID2021-124440NB-I00. E.\ Miranda is supported by the Catalan Institution for Research and Advanced Studies via an ICREA Academia Prize 2021. D.\ Peralta-Salas is supported by the grants RED2022-134301-T and PID2022-136795NB-I00 funded by MCIN/AEI/10.13039/501100011033.  S.\ Dyhr and E.\ Miranda are supported by the Spanish State Research Agency project PID2023-146936NB-I00 funded by MICIU/AEI/10.13039/501100011033. \'A.\ Gonz\'alez-Prieto and D.\ Peralta-Salas are supported by the project CEX2023-001347-S. 
The last three authors are supported by the Bilateral AEI-DFG project: Celestial Mechanics, Hydrodynamics, and Turing Machines with reference codes PCI2024-155042-2 and PCI2024-155062-2. All authors are partially supported by the project “Computational, dynamical and geometrical complexity in fluid dynamics” (COMPLEXFLUIDS), Ayudas Fundación BBVA a Proyectos de Investigación Científica 2021.
This work is supported by the Spanish State Research Agency, through the Severo Ochoa and María de Maeztu Program for Centers and Units of Excellence in R\&D (CEX2020-001084-M). We thank CERCA Programme/Generalitat de Catalunya for institutional support.}

\begin{document}
\begin{abstract}
	 In this paper, we study the Chern--Hamilton energy functional on compact cosymplectic manifolds, fully classifying in dimension~$3$ those manifolds admitting a critical compatible metric for this functional.
	This is the case if and only if either the manifold is co-Kähler or if it is a mapping torus of the 2-torus by a hyperbolic toral automorphism and equipped with a suspension cosymplectic structure.
	Moreover, any critical metric has minimal energy among all compatible metrics.
	We also exhibit examples of manifolds with first Betti number $b_1 \geq 2$ admitting cosymplectic structures, but such that no cosymplectic structure admits a critical compatible metric.
\end{abstract}

\maketitle

\section{Introduction}
\label{sec:introduction}

The pursuit of optimal metrics through energy minimization has long been a central theme in differential geometry, linking the structure of geometric spaces to deeper analytic and topological properties of manifolds.
A significant milestone in this direction is Calabi’s introduction of extremal metrics in Kähler geometry~\cite{calabi1982extremal}, defined as critical points of the Calabi functional, which measures the deviation of a metric’s scalar curvature from being constant  within a fixed Kähler class.
Earlier, in the 1950s, Calabi formulated what became known as the Calabi conjecture~\cite{calabi1954space, calabi1957kahler,tian2000canonical}, which was solved by Yau in 1978~\cite{yau1977calabi}.
The conjecture concerns the existence of Kähler metrics with prescribed Ricci curvature on compact Kähler manifolds, and Yau’s solution established the existence of Calabi–Yau metrics—Ricci-flat Kähler–Einstein metrics that arise naturally as minimizers of Calabi’s functional.
In parallel, the Yamabe problem~\cite{yamabe1960deformation, schoen1984conformal} seeks canonical metrics within a conformal class by minimizing a curvature-related functional, ultimately yielding metrics of constant scalar curvature.

Following this tradition, Chern and Hamilton~\cite{chernRiemannianMetricsAdapted1985} studied an energy functional on the space of metrics compatible with a contact structure.
They conjectured that, when restricted to contact \( 3 \)-manifolds whose Reeb field induces a Seifert fibration, energy-minimizing metrics always exist.
Since then, the study of this problem has attracted significant attention.
Their conjecture was proved to be true~\cite{rukimbiraChernHamiltonsConjectureKcontactness}, and a generalized version was later posed~\cite{hozooriDynamicsTopologyConformally2020}, asking whether minimizing metrics exist on all contact \( 3 \)-manifolds.
Recent work~\cite{hozooriAnosovContactMetrics2023,mitsumatsuExistenceCriticalCompatible2025} has disproved this generalized conjecture.
In fact, when studying the functional on general contact manifolds, even the existence of compatible metrics which are just critical points of the functional is very restricted, and it is even possible to classify all contact \( 3 \)-manifolds that admit critical compatible metrics~\cite{hozooriAnosovContactMetrics2023,mitsumatsuExistenceCriticalCompatible2025}.

In this article we will study the Chern-Hamilton problem in the cosymplectic setting, particularly analyzing critical compatible metrics on cosymplectic manifolds.
The cosymplectic case can be understood as a degenerate geometric structure that lies at the ``boundary'' of contact structures.
In fact, both contact and cosymplectic structures can be regarded as opposite extreme cases of confoliations~\cite{eliashberg1998confoliations,geiges2008}.

A cosymplectic structure on a manifold $M$ of dimension~\( 2n + 1 \geq 3 \) consists of a closed \( 1 \)-form \( \alpha \) and a closed \( 2 \)-form \( \beta \) such that \( \alpha \wedge \beta^n \) is a volume form.
It uniquely defines the so-called Reeb vector field \( \Reebfield \) characterized by \( \alpha(\Reebfield) = 1 \) and \( \beta(\Reebfield, \emptyslot) = 0 \).
Given a fixed cosymplectic structure, one can study Riemannian metrics $g$ that are compatible with it, in a similar vein to the contact case; see \cref{def:compatible-metric} for details.

Metrics compatible with a given cosymplectic structure always exist, but they are far from unique.
In this direction, the functional which Chern and Hamilton~\cite{chernRiemannianMetricsAdapted1985} studied for contact structures aimed to minimize the \emph{scalar torsion} \( \norm{\Liederivative{\Reebfield}g}^2 \), obtained from the Lie derivative $\Liederivative{\Reebfield}g$ of a compatible metric in the direction of the Reeb field.
Adapting these ideas to the framework of cosymplectic manifolds, we propose the following version of their functional:
\begin{equation*}
	E(g)
	:=
	\int_M
	\norm{\Liederivative{\Reebfield} g}^2
	\,
	\alpha \wedge \beta^n
	,
\end{equation*}
where the functional is defined on the space of metrics \( g \) compatible with the fixed cosymplectic structure \( (\alpha, \beta) \) on \( M \).
We will refer to this as the \emph{Chern--Hamilton energy functional} in the setting of cosymplectic manifolds.
An obvious class of metrics that minimize this functional are those for which the Reeb vector field is Killing, as they have zero energy.
In dimension~\( 3 \) this is equivalent to the manifold being \emph{co-Kähler}, that is, it is a cosymplectic manifold satisfying certain integrability condition~\cite{goldbergIntegrabilityAlmostCosymplectic1969}, see \cref{sec:preliminaries-metrics} for details.

Our goal in this work is to study \emph{critical compatible metrics}, that is, compatible metrics on a cosymplectic manifold that are critical points of the Chern--Hamilton energy functional.
Analogously to the classification of contact $3$-manifolds that admit critical compatible metrics~\cite{hozooriAnosovContactMetrics2023,mitsumatsuExistenceCriticalCompatible2025}, in this paper we fully classify which \( 3 \)-dimensional cosymplectic manifolds admit critical compatible metrics.
Our main result is the following:

\begin{theorem}
	Let \( (M, \alpha, \beta) \) be a compact, connected, \( 3 \)-dimensional cosymplectic manifold and let \( \Reebfield \) be its associated Reeb vector field.
	Then a compatible metric \(g\) is critical for the Chern--Hamilton energy functional if and only if either
   \begin{enumerate}
	   \item \( \Liederivative{\Reebfield}g = 0 \), so the structure is co-Kähler, or,
	   \item \(M\) is cosymplectomorphic to a suspension cosymplectic structure on a symplectic mapping torus of the $2$-torus by a hyperbolic toral automorphism and is equipped with the metric described in \cref{ex:critical-metric-mapping-torus}.
	   In this case \(\norm{\Liederivative{\Reebfield}g}^2 \) is a positive constant.
   \end{enumerate}
   Additionally, any critical metric minimizes the energy among all compatible metrics.
   \label{thm:classification}
\end{theorem}

Notice that, by \cite{bazzoniStructureCoKahlerManifolds2014}, all compact $3$-dimensional co-Kähler manifolds are the product of a compact surface with a circle up to a finite cover.
It is also worth mentioning that the manifolds of case (2) of \cref{thm:classification} actually admit several equivalent presentations, as described in \cref{ex:critical-metric-mapping-torus} and \cref{ex:critical-metric-Sol-quotient}.
In particular, the manifolds admitting critical metrics with positive torsion as well as their cosymplectic and metric structure can be described explicitly, see \cref{sec:critical-metrics-construction}.

As a consequence of our main result, we also show that there exist manifolds that, although they admit cosymplectic structures, they do not fall into either of the two types described in \cref{thm:classification} for any cosymplectic structure on these manifolds; see \cref{sec:no-critical-metrics-construction}.
In other words, there exist manifolds that admit cosymplectic structures but do not admit any cosymplectic structure with a critical metric.

The strategy to prove \cref{thm:classification} is the following.
First, we will show that a compatible metric $g$ is critical if and only if it satisfies the Euler--Lagrange equation
\begin{equation*}
	\nabla_{\Reebfield} \Liederivative{\Reebfield}g
	=
	0
    .
	\label{eq:critical-point-condition}
\end{equation*}
This extends to the cosymplectic setting the critical point condition for the Chern--Hamilton functional obtained by Tanno~\cite{tannoVariationalProblemsContact1989} in the contact case.
In fact, we calculate the critical point condition in slightly greater generality than is necessary for \cref{thm:classification}, see \cref{sec:functional} for the details.

A second key ingredient in the proof of \cref{thm:classification} is the use of \emph{algebraic Anosov flows}.
These are a special class of Anosov flows (see \cref{def:algebraic-anosov,def:hyperbolic-set-and-Anosov}), arising on compact smooth quotients of the Lie groups \( \modelLieGroup \) and \( \widetilde{\Liegroup{SL}}(2, \mathbb{R}) \), the universal cover of \( \Liegroup{SL}(2, \mathbb{R}) \).
In fact, up to smooth parameter change, these are the only smooth, volume-preserving Anosov flows with smooth stable and unstable directions on compact \( 3 \)-manifolds~\cite{ghysDeformationsFlotsDAnosov1992}.

In this direction, we show (\cref{thm:stable-unstable-directions-as-kernels-of-global-tensors,thm:critical-metric-Reeb-is-algebraic-Anosov}) that the critical metrics with positive torsion on cosymplectic manifolds are Anosov with smooth stable and unstable directions and that they in fact are algebraic Anosov flows of \( \modelLieGroup \) type (up to reparametrization).

To conclude this introduction, we observe several parallels between our classification results for critical metrics in the cosymplectic case and those in the contact setting, as detailed in~\cite{hozooriAnosovContactMetrics2023,mitsumatsuExistenceCriticalCompatible2025}.
In both frameworks, the scalar torsion $\norm{\Liederivative{\Reebfield}g}^2$ is everywhere constant for critical metrics, leading to the following dichotomy (see \cref{tab:contact-vs-cosymplectic}
 for reference):

\begin{itemize}
	\item When the torsion is zero, \( \Liederivative{\Reebfield}g = 0 \), it means that the Reeb vector field acts as a Killing field.
	In the contact context, such structures are termed K-contact, which in three dimensions is equivalent to being Sasakian, whereas in the cosymplectic setting they are known as co-Kähler.

	\item When the torsion is positive, the Reeb flow becomes an Anosov flow with smooth stable and unstable directions.
		According to~\cite{ghysDeformationsFlotsDAnosov1992}, this implies that the flow is algebraic Anosov (up to reparametrization).
		Thus the manifold is a compact smooth quotient of either \( \modelLieGroup \) or \( \widetilde{\Liegroup{SL}}(2, \mathbb{R}) \).
	The former aligns with the cosymplectic case classification, while the latter pertains to the contact case.
\end{itemize}

\begin{table}
	\centering
	\renewcommand{\arraystretch}{1.3}
	\caption{Comparison of the classification of critical metrics in \( 3 \) dimensions in the contact and cosymplectic settings.}
	\begin{tabular}{l l l}
		\toprule
			& Contact
				& Cosymplectic
		\\
		\midrule
		\( \Liederivative{\Reebfield}g = 0 \)
			& Sasakian
				& Co-Kähler
		\\\midrule
		\multirow{ 2}{*}{\( \norm{\Liederivative{\Reebfield}g}^2 > 0 \)}
			& \multicolumn{2}{p{4cm}}{\centering Algebraic Anosov}\\
			& \( \widetilde{\Liegroup{SL}}(2, \mathbb{R}) \) type
				& \( \modelLieGroup \) type
		\\
		\bottomrule
	\end{tabular}
	\label{tab:contact-vs-cosymplectic}
\end{table}

These parallels suggest the potential existence of a geometric structure simultaneously generalizing contact and cosymplectic structures for which critical metrics exist exactly if the Reeb vector field is either Killing or its flow is a reparametrization of an algebraic Anosov flow.
A natural candidate for this general framework are the \( \Reebfield \)-invariant almost cosymplectic structures considered in \cref{thm:Euler--Lagrange}.

\section{Preliminaries}
\label{sec:preliminaries}

In this \namecref{sec:preliminaries} we give a more detailed overview of the objects to be studied and mention some of their properties that we will need.

\begin{definition}
	A cosymplectic manifold is a \( (2n + 1) \)-dimensional manifold \( M \) equipped with a closed \( 1 \)-form \( \alpha \) and a closed \( 2 \)-form \( \beta \) such that \( \alpha \wedge \beta^{n} \) is a volume form.
	The unique vector field \( \Reebfield \) satisfying that
	\begin{equation*}
		\alpha(\Reebfield) = 1
		\quad\text{and}\quad
		\beta(\Reebfield, \emptyslot) = 0
	\end{equation*}
is called the \emph{Reeb vector field} and its flow the \emph{Reeb flow}.
	\label{def:cosymplectic}
\end{definition}
A smooth map between cosymplectic manifolds \( f \colon (M_1,\alpha_1,\beta_1) \to (M_2,\alpha_2,\beta_2) \) that satisfies \( f^* \alpha_2 = \alpha_1 \) and \( f^* \beta_2 = \beta_1 \) is said to be \emph{cosymplectic}.
A cosymplectic diffeomorphism is called a \emph{cosymplectomorphism}.

\subsection{Examples of cosymplectic manifolds}
\label{sec:preliminaries-cosymplectic-manifolds}

A simple class of (open) cosymplectic manifolds is formed by products of symplectic manifolds with intervals.
Let \( (S, \omega) \) be a symplectic manifold, \( I \) an (open) interval, and \( \dtFactor \neq 0 \) a real constant.
Then \( (S \times I, \dtFactor \dif t, p_1^* \omega) \) is cosymplectic with Reeb vector field \( \dtFactor^{-1} \partial_{t} \), where \( t \) is the coordinate on \( I \) and \( p_1 \colon S \times I \to S \) the projection.

For another construction, take a symplectomorphism \( f \colon S \to S \).
Then we have the \emph{symplectic mapping torus} \( S_f := S \times \mathbb{R}/(p, t+1)\sim(f(p), t) \).
The cosymplectic structure \( (\dtFactor \dif t, p_1^*\omega) \) on \( S \times \mathbb{R} \) descends to the quotient since \( f \) is a symplectomorphism.
The Reeb flow is the suspension flow of \( f \) (up to the scaling by \( \dtFactor^{-1} \)).

Any matrix \( L \in \Liegroup{SL}(2, \mathbb{Z}) \) generates a diffeomorphism of the torus \( \mathbb{T}^{2} = \mathbb{R}^{2}/\mathbb{Z}^{2} \), which is a symplectomorphism of $\mathbb{T}^2$ with respect to the standard symplectic structure.
In case \( L \) has no eigenvalue of absolute value one, we call the resulting map a \emph{hyperbolic toral automorphism}.
The mapping tori $\mathbb{T}^2_L$ glued by these maps form one of the two types of cosymplectic manifolds admitting critical metrics in \cref{thm:classification}.

\subsection{Algebraic Anosov flows}
\label{sec:algebraic-anosov}

Algebraic Anosov flows are a class of flows on certain homogeneous spaces.
The Lie groups acting on these spaces are \( \widetilde{\Liegroup{SL}}(2, \mathbb{R}) \) and \( \modelLieGroup \), the simply connected Lie groups with Lie algebras respectively \( \Liealgebra{sl}(2, \mathbb{R}) \) and \( \modelLieAlgebra \).
These are both three-dimensional, for \( \Liealgebra{sl}(2, \mathbb{R}) \) one can choose a basis \( (Y, X_{+}, X_{-}) \) such that \( [Y, X_{\pm}] = \pm X_{\pm} \) and \( [X_{+}, X_{-}] = Y\).
The Lie algebra \( \modelLieAlgebra \) has a basis \( (Y, X_{+}, X_{-}) \) with Lie bracket \( [Y, X_{\pm}] = \pm X_{\pm} \) and \( [X_{+}, X_{-}] = 0 \).

In both cases, the element \( Y \) (or any rescaling of it) defines a \( 1 \)-parameter subgroup \( t \mapsto \exp tY \) which by right-multiplication gives rise to a flow \( \Phi^{t} \) on the Lie group, that is, \( \Phi^{t}(g) = g\exp(tY) \).
Let \( \Gamma \) be a subgroup of \( G \).
The flow then descends to give a flow \( \Phi^{t} (\Gamma g) = \Gamma g\exp(tY) \) on any homogeneous space \( \Gamma \backslash G \) of the Lie group.
This leads to the following key notion in our work (for the definition of Anosov flow, please refer to \cref{def:hyperbolic-set-and-Anosov}).

\begin{definition}
	An Anosov flow on a manifold is said to be \emph{algebraic} if it is covered by one of the flows \( \Phi^{t} \) on either \( \widetilde{\Liegroup{SL}}(2, \mathbb{R}) \) or \( \modelLieGroup \).
	\label{def:algebraic-anosov}
\end{definition}

In particular, if \( G \) is either \( \modelLieGroup \) or \( \widetilde{\Liegroup{SL}}(2, \mathbb{R}) \), the flows \( \Phi^{t} \) on Lie group quotients \( \Gamma \backslash G \) for \( \Gamma \) a discrete, cocompact subgroup of \( G \) are algebraic Anosov flows.
More generally, \( \Phi^{t} \) also descends to become an algebraic Anosov flow on compact quotients of \( G \) by isometries.

\begin{remark}
	The Lie group \( \modelLieGroup \) can be explicitly described as the semidirect product \( \mathbb{R} \ltimes \mathbb{R}^{2} \), where an element \( t \in \mathbb{R} \) acts on \( \mathbb{R}^2 \) by the diagonal matrix with diagonal \( (e^{t}, e^{-t}) \).

	The \( \modelLieGroup\) group is also known by other names, such as the Poincaré group of \( 2 \)-dimensional spacetime, that is, it is the group of isometries of the plane equipped with the inner product with signature \( (+, -) \).
	It is also called \( \Liegroup{AO}(1, 1) \) (for affine orthogonal) and \( E(1, 1) \) (for Euclidean).
	As a matrix Lie group it can be represented as the set of matrices of the form
	\begin{equation*}
		\begin{pmatrix}
			e^{t}
				& 0
					& a
			\\
			0
				& e^{-t}
					& b
			\\
			0&	0&	1
		\end{pmatrix}
		,
	\end{equation*}
	where \( t, a, b \in \mathbb{R} \).

	Similarly its Lie algebra \( \modelLieAlgebra \) is also known as \( \mathfrak{p}(1, 1) \), the Poincaré Lie algebra, and is the Lie algebra VI\textsubscript{0} in the Bianchi classification.
	\label{rem:Sol-definition}
\end{remark}

\subsection{Compatible metrics and the Chern--Hamilton functional}
\label{sec:preliminaries-metrics}

\begin{definition}\label{def:compatible-metric}
	A Riemannian metric \( g \) is called \emph{compatible} with the cosymplectic structure $(\alpha,\beta)$ if there exists a \( (1, 1) \)-tensor field \( \phi \) such that
	\begin{equation}
		\phi^2 = -\identityOperator + \alpha \tensorproduct \Reebfield,
		\quad
		\beta = g(\emptyslot, \phi \emptyslot),
		\quad\text{and}\quad
		\alpha = g(\Reebfield, \emptyslot)
		.
		\label{eq:def:compatible-metric-direct}
	\end{equation}
	In dimension~\( 3 \) (i.e., when \( n = 1 \)), this is equivalent to a metric \( g \) satisfying that
	\begin{equation}
		\HodgeStar_g \alpha = \beta,
		\quad\text{and}\quad
		g(R, R) = 1
		,
		\label{eq:compatible-metric-Hodge-star-definition}
	\end{equation}
	where \( \HodgeStar_g: \Omega^1(M) \to \Omega^2(M) \) denotes the Hodge star operator with regard to the metric \( g \).
\end{definition}

The proof that the two definitions are equivalent in dimension~\( 3 \) is elementary, so we omit it.
In dimension~\( 2n + 1 > 3 \) a metric can satisfy \cref{eq:compatible-metric-Hodge-star-definition} (with \( \beta \) replaced by \( \beta^n \)) without being compatible with the cosymplectic structure.
Note that, given a compatible metric \( g \), the requirement \( \beta = g(\emptyslot, \phi \emptyslot) \) uniquely determines \( \phi \).
On the other hand, given a cosymplectic structure many compatible metrics exist.
One can, for example, construct compatible metrics by using a polar decomposition on arbitrary metrics.

Let $(M,\alpha,\beta)$ be a cosymplectic manifold and $g$ be a compatible Riemannian metric with associated \( (1, 1) \)-tensor field \( \phi \).
If $\Reebfield$ is the Reeb field defined by the cosymplectic structure, the following properties are well known, see e.g.~\cite{blairRiemannianGeometryContact2010}:
\begin{equation}
	\phi(\Reebfield) = 0
	,
	\quad
	\alpha \circ \phi = 0,
	\quad
	\alpha
	=
	g(\Reebfield, \emptyslot)
	,
	\quad
	\norm{\Reebfield} = 1
	\quad\text{and}\quad
	\nabla_{\Reebfield}\Reebfield = 0,
	\label{eq:compatible-metric-properties}
\end{equation}
where the covariant derivative is taken for the Levi--Civita connection $\nabla$ associated to the metric $g$.
It is also easy to check that any compatible metric has the form 
\begin{equation}
	g = g(\phi\emptyslot, \phi\emptyslot) + \alpha \tensorproduct \alpha\,.
	\label{eq:g=g(phiphi)+alphaalpha}
\end{equation}
Note that \( (\phi, \alpha, \Reebfield, g) \) forms an almost contact metric structure satisfying extra conditions coming from the cosymplectic structure.
The covariant derivative of \( \phi \) can be calculated as
\begin{equation}
	2g\bigl(
		(\nabla_X \phi)Y,
		Z
	\bigr)
	=
	g([\phi, \phi](Y, Z), \phi X)),
	\label{eq:phi-covariant-derivative}
\end{equation}
where \( [\phi, \phi] \) is the Nijenhuis bracket of \( \phi \) with itself, see~\cite[Lemma 6.1]{blairRiemannianGeometryContact2010}.

We remark that a cosymplectic manifold with a compatible metric is sometimes referred to as \emph{almost co-Kähler} in the literature.
If the Nijenhuis bracket \( [\phi, \phi] \) vanishes, it is called \emph{co-Kähler}.
In dimension~\( 3 \), an almost co-Kähler manifold is co-Kähler if and only if the Reeb vector field is Killing, i.e., \( \Liederivative{\Reebfield}g = 0 \)~\cite{goldbergIntegrabilityAlmostCosymplectic1969}.

\begin{remark}
	We observe that what we call co-Kähler is also called cosymplectic in the literature (e.g., by Blair), while our cosymplectic might then be called almost cosymplectic (as in~\cite{blairRiemannianGeometryContact2010,goldbergIntegrabilityAlmostCosymplectic1969}).
	We follow the nomenclature used originally by Libermann~\cite{libermann1959} and more recently by, for example, Li~\cite{liTopologyCosymplecticCoKahler2008} and Guillemin, Miranda, and Pires~\cite{guilleminCodimensionOneSymplectic2011}.
\end{remark}

An important object when studying almost contact metric structures, and thus in particular compatible metrics on cosymplectic manifolds, is the tensor field \( h := \frac{1}{2} \Liederivative{\Reebfield} \phi \).
We recall here some of its properties.

\begin{lemma}
	Let \( (M, \alpha, \beta) \) be a cosymplectic manifold with a compatible metric \( g \) and associated \( (1, 1) \)-tensor field \( \phi \).
	Then \( h \) is symmetric and satisfies the following identities:
	\begin{equation}
		h(\Reebfield) = 0
		,\quad
		h \phi + \phi h = 0
		,\quad
		\Liederivative{\Reebfield}g = 2g(\emptyslot, h \phi \emptyslot)
		,
		\quad\text{and}\quad
		\nabla \Reebfield = h \phi
		.
		\label{eq:LRphi-properties}
	\end{equation}
	Moreover, \( h \) is orthogonally diagonalizable at every point of \( M \) with eigenvalues \( (0, \mu, -\mu) \) where \( \mu = 2^{-3/2} \norm{\Liederivative{\Reebfield}g} \).
	\label{thm:LRphi-properties}
\end{lemma}

\begin{proof}
	Since \( \phi \) is antisymmetric, \( h \) is symmetric.
	One has \( 0 = \Liederivative{\Reebfield}(\phi(\Reebfield)) = 2h(\Reebfield) \) and \( 0 = 
	\Liederivative{\Reebfield}(\phi^2) = 2(h \phi + \phi h) \).
	The third identity of \cref{eq:LRphi-properties} follows from taking the Lie derivative in the Reeb direction of the identity \cref{eq:g=g(phiphi)+alphaalpha}, while the last identity follows from setting \( Y = \Reebfield \) in \cref{eq:phi-covariant-derivative}.

	The pointwise orthogonal diagonalizability is due to \( h \) being symmetric.
	As we have already proven, the Reeb vector field is an eigenvector field with eigenvalue \( 0 \), and if \( v \) is an eigenvector with eigenvalue \( \mu \), \( \phi v \) is an eigenvector with eigenvalue \( -\mu \).
	This together with \( \Liederivative{\Reebfield}g = 2g(\emptyslot, h \phi \emptyslot) \) and the fact that \( \phi \) is an isometry on \( \ker\alpha \) implies that \( \norm{\Liederivative{\Reebfield}g}^{2} = 4 \norm{h}^{2} = 8\mu^{2} \).
\end{proof}

The energy functional we study is the following.
\begin{definition}
	Let \( (M, \alpha, \beta) \) be a \( (2n+1) \)-dimensional cosymplectic manifold.
	The Chern--Hamilton energy functional adapted to this situation is
	\begin{equation}
		E(g)
		:=
		\int_M
		\norm{\Liederivative{\Reebfield} g}^2
		\,
		\alpha \wedge \beta^n
		\label{eq:Chern--Hamilton-functional}
		.
	\end{equation}
	The domain of the functional is the space of compatible metrics.
	\label{def:Chern--Hamilton-functional}
\end{definition}

We say that a metric is \emph{critical} if it is a critical point of the Chern--Hamilton functional in the variational sense, i.e., \( g_{0} \) is a critical metric if for all smooth curves \( g(t) \) of compatible metrics with \( g(0) = g_{0} \), 
\begin{equation*}
	\odfrac{}{t}\bigg\vert_{t=0} E(g(t)) = 0
	.
\end{equation*}

\subsection{Equivalent models of hyperbolic mapping tori}
\label{sec:critical-metrics-construction}

In this \namecref{sec:critical-metrics-construction}, we shall study different models for the cosymplectic manifolds arising as mapping tori of \( \mathbb{T}^2 \) by a hyperbolic toral automorphism $L$ in terms of quotients of \( \modelLieGroup \), and describe compatible metrics on them.
As we will show in \cref{sec:critical-metrics-Anosov}, all the constructed metrics in this \namecref{sec:critical-metrics-construction} are critical and, reciprocally, all critical metrics are of the types constructed below.

\begin{example}
	Let \( L \) be a hyperbolic toral automorphism (a matrix in \( \Liegroup{SL}(2, \mathbb{Z}) \) with no eigenvalues of absolute value one) and diagonalize it as \( L w_{\pm} = \lambda^{\pm 1} w_{\pm} \).
	Let us choose \( \lambda \) to be the (automatically real) eigenvalue with \( \abs{\lambda} > 1 \).
	Choose a \( V > 0 \) and normalize the \( w_{\pm} \) by imposing that \( (V \dif x \wedge \dif y)(w_{+}, w_{-}) = 1 \).
	The eigenvectors extend to vector fields on \( \mathbb{T}^{2} \times \mathbb{R} \).
	Define \( v_{\pm} := \abs{\lambda}^{\mp t} w_{\pm} \), \( t \) being the coordinate on the \( \mathbb{R} \) factor.

	Choosing a parameter \( \dtFactor > 0 \) gives a cosymplectic structure \( (\dtFactor \dif t, V \dif x \wedge \dif y) \) on \( \mathbb{T}^{2} \times \mathbb{R} \), and declaring the frame \( (\Reebfield = \dtFactor^{-1} \partial_{t}, v_{+}, v_{-}) \) to be orthonormal yields a compatible metric \( g \).
	By changing coordinates on \( \mathbb{T}^{2} \) to \( x_{\pm} \) such that the eigenvectors become \( w_{\pm} = \partial_{x_{\pm}} \) we can write the metric as \( \dif t^2 + \lambda^{2t} \dif x_{+}^{2} + \lambda^{-2t} \dif x_{-}^{2} \).
	The cosymplectic structure descends to the mapping torus \( \mathbb{T}^{2}_{L} \) (see \cref{sec:preliminaries-cosymplectic-manifolds}).
	To see that the metric also descends, recall that this is equivalent to it being invariant under the map \( F \colon (p, t+1) \mapsto (Lp, t) \).
	We have \( \dif F(\Reebfield) = \Reebfield \) and \( \dif F(w_{\pm}) = \lambda^{\pm 1} w_{\pm} \), so \( F^* \dif x_{\pm} = \lambda^{\pm 1} \dif x_{\pm} \), and we see that
	\begin{equation*}
		F^*(g_{(Lp, t)})
		=
		\dif t^2
		+
		\lambda^{2t} (\lambda \dif x_{+})^{2}
		+
		\lambda^{-2t} (\lambda^{-1} \dif x_{-})^{2}
		=
		g_{(p, t+1)}
		.
	\end{equation*}
	The construction thus gives a cosymplectic structure with a compatible metric on \( \mathbb{T}^{2}_{L} \), for which a direct calculation shows that \( \norm{\Liederivative{\Reebfield}g}^2 = 8(\dtFactor^{-1} \ln \abs{\lambda})^{2} \) and \( E(g) = 8V\dtFactor^{-1} (\ln \abs{\lambda})^{2} \).
	\label{ex:critical-metric-mapping-torus}
\end{example}

Note that if \( \lambda > 0 \) the \( v_{\pm} \) become smooth vector fields on \( \mathbb{T}^{2}_{L} \), otherwise they will only be defined locally with a global sign ambiguity coming from the quotient \( \mathbb{T}^{2} \times \mathbb{R} \to \mathbb{T}^2_L \).

\begin{remark}
	In \cref{ex:critical-metric-mapping-torus} a choice of the eigenvectors \( w_{\pm} \) is made.
	Since the pair \( (w_{+}, w_{-}) \) is normalized, the only freedom is in scaling \( (w_{+}, w_{-}) \) to \( (c w_{+}, c^{-1}w_{-}) \) by \( c \neq 0 \).
	The resulting metric has the form
	\begin{equation*}
		\dif t^{2} + c^{-2}\lambda^{2t} \dif x_{+}^{2} + c^{2}\lambda^{-2t} \dif x_{-}^2
		.
	\end{equation*}
	However, the coordinate change \( t \mapsto t + (2 \ln \lambda)^{-1}\ln c \) preserves the cosymplectic structure and puts the metric on the form from \cref{ex:critical-metric-mapping-torus} (i.e., with \( c = 1 \)).
	In conclusion, the construction does not have any hidden choices.
	\label{rem:HTA-mapping-torus-critical-metric-freedom}
\end{remark}

Consider now the Lie group \( \modelLieGroup \) (\cref{rem:Sol-definition}).
It admits a left-invariant cosymplectic structure with compatible metric, obtained by defining these structures on its Lie algebra and then translating them to the rest of the group.
Explicitly, the construction is as follows.

\begin{example}
	Take coordinates \( (t, x_{+}, x_{-}) \) on \( \modelLieGroup = \mathbb{R} \ltimes \mathbb{R}^2 \).
	Consider the vectors obtained by restricting \( (\partial_{t}, \partial_{x_{\pm}}) \) to \( \modelLieAlgebra \), the tangent space of \( \modelLieGroup \) at the identity.
	By left-multiplication these vectors extend to form a frame of left-invariant vector fields on \( \modelLieGroup \) given by \( (\partial_{t}, e^{\pm t} \partial_{x_{\pm}}) \).
	Notice that these vector fields form the basis of \( \modelLieAlgebra \) as described in \cref{rem:Sol-definition}: \( [\partial_{t}, e^{\pm t} \partial_{x_{\pm}}] = \pm e^{\pm t} \partial_{x_{\pm}} \) and \( [e^{t} \partial_{x_{+}}, e^{-t} \partial_{x_{-}}] = 0 \).
	Similarly, the \( 1 \)-forms \( (\dif t, e^{\pm t}\dif x_{\pm}) \) form a global, left-invariant coframe.
	
	For any \( \mu \neq 0 \), \( (\mu \dif t, \dif x_{+} \wedge \dif x_{-}) \) is a cosymplectic structure and \( \mu^{2} \dif t^{2} + \dif x_{+}^{2} + \dif x_{-}^{2} \) is a compatible metric on \( \modelLieAlgebra \).
	The left-invariant extensions of these objects to the whole Lie group \( \modelLieGroup \) gives the cosymplectic structure \( (\mu \dif t, \dif x_{+} \wedge \dif x_{-}) \) and the compatible metric
	\begin{equation*}
		g
		=
		\mu^{2}
		\dif t^2
		+
		e^{-2t} \dif x_{+}^{2}
		+
		e^{2t} \dif x_{-}^{2}
		.
	\end{equation*}
	They are related through the left-invariant \( (1, 1) \)-tensor field
	\begin{equation*}
		\phi
		=
		e^{2t} \dif x_{-} \tensorproduct \partial_{x_{+}}
		-
		e^{-2t} \dif x_{+} \tensorproduct \partial_{x_-}
		.
	\end{equation*}
	The metrics \( g = \mu^{2} \dif t^2 + ce^{-2t}\dif x_{+}^{2} + ce^{2t} \dif x_{-}^{2} \) are also compatible for any \( c > 0 \).
	However, the map \( (t, x_{+}, x_{-}) \mapsto (t, c^{-1}x_{+}, c x_{-}) \) is a cosymplectic isometry to the model with \( c = 1 \), removing the need to consider other values of \( c \).
	\label{ex:compatible-metric-Sol}
\end{example}

Let \( \Aut(\modelLieGroup) \) denote the group of diffeomorphisms of \( \modelLieGroup \) preserving both the cosymplectic and Riemannian structure defined above (depending implicitly on the choice of \( \mu \)).

\begin{example}
	If \( \Gamma \) is a subgroup of \( \Aut(\modelLieGroup) \) such that \( M := \Gamma\backslash\modelLieGroup \) is compact and the quotient map is a covering, the cosymplectic structure with compatible metric from \cref{ex:compatible-metric-Sol} descends to \( M \).
	Moreover, the resulting Reeb vector field is algebraic Anosov (up to reparametrization by \( \mu \)).
	\label{ex:critical-metric-Sol-quotient}
\end{example}

These two (families of) constructions, \cref{ex:critical-metric-mapping-torus} and \cref{ex:critical-metric-Sol-quotient}, actually yield the same manifolds and cosymplectic and metric structures, see \cref{thm:Sol-covered=mapping-torus}.
The easiest way of proving this will combine the fact that the metrics constructed in \cref{ex:critical-metric-mapping-torus} are critical with some properties of critical metrics.
For this reason we postpone the proof until \cref{sec:Sol-and-mapping-torus-models-equivalent}.

\section{The critical point condition}
\label{sec:functional}

In this \namecref{sec:functional} we calculate the critical point condition for the  Chern--Hamilton energy functional.
This was done by Tanno~\cite{tannoVariationalProblemsContact1989} for contact manifolds.
Our proof is a generalization of that proof to a larger class of structures that contain both contact and cosymplectic manifolds.

We say that a \( (2n+1)\)-dimensional manifold \( M \) equipped with a \( 1 \)-form \( \alpha \) and a \( 2 \)-form \( \beta \) is \emph{almost cosymplectic} if \( \alpha \wedge \beta^n \) is a volume form.
This generalizes both contact manifolds (where \( \beta = \dif \alpha \)) and cosymplectic manifolds (for which \( \alpha \) and \( \beta \) are closed). The associated Reeb field is defined as the unique vector field  $\Reebfield$ such that $\alpha(\Reebfield)=1$ and $\beta(\Reebfield,\cdot)=0$.

We say that a Riemannian metric $g$ is compatible with an almost cosymplectic structure if there exists a \( (1, 1) \)-tensor field \( \phi \) such that
\begin{equation}
	\phi^2 = -\identityOperator + \alpha \tensorproduct \Reebfield,
	\quad
	\beta = g(\emptyslot, \phi \emptyslot),
	\quad\text{and}\quad
	\alpha = g(\Reebfield, \emptyslot)
	.
	\label{eq:compatible-metric-almost-cosymplectic}
\end{equation}
Note that if \( \phi \) exists it is uniquely defined by the requirement \( \beta = g(\emptyslot, \phi \emptyslot) \).
This definition of compatible metric generalizes the definitions of compatible metrics for contact and cosymplectic manifolds. The Chern--Hamilton energy functional is defined exactly as in the cosymplectic case (\cref{def:Chern--Hamilton-functional}).

We will in fact not prove the critical point condition for arbitrary almost cosymplectic manifolds but rather for a specific subset that still contains both cosymplectic and contact manifolds.
Let us be more precise.
\begin{definition}
	An almost cosymplectic structure \( (\alpha, \beta) \) with Reeb vector field \( \Reebfield \) is \emph{\( \Reebfield \)-invariant} if
	\begin{equation}
		\Liederivative{\Reebfield}\alpha = 0
		\quad\text{and}\quad
		\Liederivative{\Reebfield}\beta = 0.
		\label{eq:R-invariant-almost-cosymplectic}
	\end{equation}
	A metric \( g \) with associated \( (1, 1) \)-tensor field \( \phi \) is compatible with an \( \Reebfield \)-invariant almost cosymplectic structure if it is compatible in the sense of \cref{eq:compatible-metric-almost-cosymplectic} and also satisfies that
	\begin{equation}
		(\dif \alpha)(\phi \emptyslot, \emptyslot) + (\dif \alpha)(\emptyslot, \phi \emptyslot) = 0
		\label{eq:R-invariant-almost-cosymplectic-compatible-metric}
	\end{equation}
	\label{def:R-invariant-almost-cosymplectic}
\end{definition}
Since \( \phi \) is uniquely defined by \( g \), the criterion \cref{eq:R-invariant-almost-cosymplectic-compatible-metric} really is a criterion on the metric.
Note also that these conditions on the almost cosymplectic structure and the compatible metric are satisfied for both contact and cosymplectic structures:
on a contact manifold, \( \dif \alpha = \beta = g(\emptyslot, \phi \emptyslot)\) for which \( \phi \) is antisymmetric whereas for cosymplectic manifolds we have \( \dif \alpha = 0 \).

The main result of this \namecref{sec:functional} is the following \namecref{thm:Euler--Lagrange} that tells us the Euler--Lagrange equation for the Chern--Hamilton functional.
In its statement, we denote by \( \difalphasharp \) the \( (1,1) \)-tensor field which is the metric dual of \( \dif\alpha \), i.e., \( g(\emptyslot, \difalphasharp\emptyslot) = \dif\alpha \).
\begin{proposition}
	Let \( (M, \alpha, \beta) \) be a compact \( \Reebfield \)-invariant almost cosymplectic manifold.
	Then, a compatible metric \( g \) is critical for the Chern--Hamilton functional if and only if 
	\begin{equation*}
		\nabla_\Reebfield \Liederivative{\Reebfield} g
		-
		2 (\Liederivative{\Reebfield} g)
		\bigl(
			\emptyslot,
			\difalphasharp \emptyslot
		\bigr)
		=
		0\,.
	\end{equation*}
	\label{thm:Euler--Lagrange}
\end{proposition}

For contact manifolds \( \dif \alpha = \beta =  g(\emptyslot, \phi \emptyslot) \) implies that $\difalphasharp=\phi$.
Meanwhile on a cosymplectic manifold \( \difalphasharp \) vanishes because \( \dif\alpha = 0 \), so the critical point condition becomes simpler:

\begin{corollary}
	Let \( (M, \alpha, \beta) \) be a compact cosymplectic manifold and \( g \) a compatible metric.
	Then \( g \) is critical for the Chern--Hamilton functional if and only if
	\begin{equation*}
		\nabla_\Reebfield \Liederivative{\Reebfield} g
		=
		0
		.
	\end{equation*}
	\label{thm:Euler--Lagrange-cosymplectic}
\end{corollary}

The only change in the critical point condition for almost cosymplectic structures when compared with the contact case consists in replacing \( \phi \) by \( \difalphasharp \). Accordingly, the general argument used by Tanno to calculate the critical point condition in the contact case~\cite{tannoVariationalProblemsContact1989} still works here.

The proof of the critical point condition contains two steps:
first, the first variation of the functional is calculated to obtain a sufficient criterion for a metric to be critical.
Then it is shown that this criterion is necessary. The full details of the computations, which contain the needed generalizations and explain the appearance of the \( \difalphasharp \) term, are provided in \cref{sec:first-variation-calculation}.

First, we state the first variation of the Chern--Hamilton energy functional, which is proved in the aforementioned appendix.

\begin{lemma}
	Let \( (M, \alpha, \beta) \) be a compact \( \Reebfield \)-invariant almost cosymplectic manifold and let \( g_{t} \) be a curve of compatible metrics with \( \phi_{t} \) the associated curve of \( (1, 1) \)-tensor fields.
	Then,
	\begin{equation*}
		\odfrac{}{t}
			E(g_{t})
		=
		2\int_M
		g_{t}\Bigl(
			2
			(\Liederivative{R}g_{t})
			\bigl(
				\emptyslot,
				\difalphasharp \emptyslot
			\bigr)
			-
			\nabla^t_\Reebfield \Liederivative{\Reebfield} g_{t}
			,
			g'_{t}
		\Bigr)
		\,\alpha\wedge\beta^n
		,
	\end{equation*}
	where \( g(\emptyslot, \difalphasharp \emptyslot) = \dif\alpha \) and \( \nabla^t \) denotes the Levi--Civita connection coming from the metric \( g_t \).
	\label{thm:Chern--Hamilton-functional-first-variation}
\end{lemma}

This \namecref{thm:Chern--Hamilton-functional-first-variation} implies that the metrics solving the equation
\[
	\nabla_\Reebfield \Liederivative \Reebfield g
	-
	2 (\Liederivative{\Reebfield} g)(
			\emptyslot
			,
			\difalphasharp \emptyslot
	) = 0
\]
are critical points of the Chern--Hamilton functional for \( \Reebfield \)-invariant almost cosymplectic manifolds.
To prove that this condition is also necessary, the following characterization of the tangent space to the space of compatible metrics is needed.
For the proof, see \cref{sec:compatible-metrics-tangent-space}.
For this result, the assumptions that \( \Liederivative{\Reebfield}\alpha \) and \( \Liederivative{\Reebfield}\beta \) vanish are not necessary, we just assume that \( (M, \alpha, \beta) \) is an almost cosymplectic manifold.

\begin{lemma}
	Let \( g_0 \) be a compatible metric on an almost cosymplectic manifold \( (M, \alpha, \beta) \).
	Then, a \( 2 \)-tensor field \( H \) is the derivative \( g'(0) \) of a curve \( g(t) \) of compatible metrics with \( g(0) = g_0 \) if and only if \( H \) is symmetric and satisfies
	\begin{equation}
		\iota_{\Reebfield} H = 0
		,
		\quad\text{ and }\quad
		H(\phi \emptyslot, \emptyslot)
		=
		H(\emptyslot, \phi \emptyslot)
		.
		\label{eq:compatible-metrics-tangent-space}
	\end{equation}
	\label{thm:compatible-metrics-tangent-space}
\end{lemma}

Now, assume that \( g_0 \) is a critical compatible metric and define
\begin{equation*}
	H
	:=
	\nabla_\Reebfield \Liederivative \Reebfield g_0
	-
	2 (\Liederivative{\Reebfield} g_0)
	(
		\emptyslot,
		\difalphasharp \emptyslot
	)
	.
\end{equation*}
The observation now is that if we create a path \( g(t) \) of compatible metrics with \( g(0) = g_0 \) and \( g'(0) = H \), then \cref{thm:Chern--Hamilton-functional-first-variation} implies that
\begin{equation*}
\begin{split}
	0
	&=
	\odfrac{}{t} \evalAt[\bigg]{t=0}{
		E\bigl( g(t) \bigr)
	}
	=
	\int_M
	g_0\Bigl(
		\nabla_{\Reebfield} \Liederivative{\Reebfield} g(0)
		-
		2 (\Liederivative{\Reebfield} g(0))\bigl(
			\emptyslot,
			\difalphasharp \emptyslot
		\bigr)
		,
		g'(0)
	\Bigr)
	\, \alpha\wedge\beta^n
	\\
	&=
	\int_M
		g_0(H, H)
	\,\alpha\wedge\beta^n
	,
\end{split}
\end{equation*}
which is only possible if \( H \) is identically \( 0 \).
The existence of such a curve is provided by \cref{thm:compatible-metrics-tangent-space} after verifying that \( H \) as defined satisfies the assumptions of that \namecref{thm:compatible-metrics-tangent-space}.
This is a straightforward calculation, and the details are included in \cref{sec:compatible-metrics-tangent-space}.

\section{Properties and structure of critical metrics}
\label{sec:critical-metrics-properties}

In this section we prove various properties satisfied by critical metrics which we will need in the proof of \cref{thm:classification}.
A first easy observation about critical metrics is that the torsion is a first integral of the Reeb vector field. We recall that the torsion is defined as the scalar quantity \( \norm{\Liederivative{\Reebfield}g}^2 \).
\begin{lemma}
	If \( g \) is a critical metric on a compact cosymplectic manifold with Reeb vector field \( \Reebfield \), the torsion  is a first integral of \( \Reebfield \).
	\label{thm:scalar-torsion-is-first-integral-for-Reeb-field}
\end{lemma}

\begin{proof}
	The Levi--Civita connection is torsion free, so the critical point condition from \cref{thm:Euler--Lagrange-cosymplectic} implies that
	\begin{equation*}
		\Reebfield
		\big(
			\norm{\Liederivative{\Reebfield}g}^2
		\big)
		=
		2
		g(
			\nabla_\Reebfield \Liederivative{\Reebfield}g
			,
			\Liederivative{\Reebfield}g
		)
		=
		0
		.
	\end{equation*}
\end{proof}
Recall the tensor field \( h = \frac{1}{2} \Liederivative{\Reebfield}g \) from \cref{thm:LRphi-properties}.

\begin{lemma}
	Let \( g \) be a compatible metric on a compact cosymplectic manifold.
	Then the metric is critical if and only if \( \nabla_{\Reebfield} h = 0 \).
	\label{thm:critical-g-equivalent-to-nablaRh=0}
\end{lemma}

\begin{proof}
	Use \cref{thm:LRphi-properties} to calculate that for any vector fields \( X \), \( Y \),
	\begin{equation}
	\begin{split}
		\frac{1}{2}(\nabla_{\Reebfield} \Liederivative{\Reebfield} g)(X, Y)
		&=
		\nabla_{\Reebfield}(
			g(X, h \phi Y)
		)
		-
		g(\nabla_{\Reebfield} X, h \phi Y)
		-
		g(X, h \phi \nabla_{\Reebfield} Y)
		\\
		&=
		g(X, (\nabla_{\Reebfield}(h\phi))(Y))
		=
		g(X, (\nabla_{\Reebfield}h)(\phi Y))
		,
	\end{split}
	\label{eq:nablaRLRg}
	\end{equation}
	since \( \nabla_{\Reebfield} \phi = 0\) by \cref{eq:phi-covariant-derivative}.
	The kernel of \( \alpha \) and \( \Reebfield \) span the tangent spaces, and on \( \ker\alpha \), \( \phi \) is an automorphism.
	At the same time, since \( h(\Reebfield) = 0 \) and \( \Reebfield \) is parallel, \( (\nabla_\Reebfield h)(\Reebfield) \) always vanishes.
	The left hand side of \cref{eq:nablaRLRg} vanishes if and only if the metric is critical (by \cref{thm:Euler--Lagrange-cosymplectic}), so this proves that the vanishing of \( \nabla_{\Reebfield} h \) is equivalent to the metric being critical.
\end{proof}

We will now focus on the \( 3 \)-dimensional case to prove \cref{thm:classification}.
Recall from \cref{thm:LRphi-properties} that the tensor field \( h = \frac{1}{2} \Liederivative{\Reebfield} \phi \) always has a local frame of eigenvector fields with eigenvalues \( (0, \mu, -\mu) \), where \( \mu = 2^{-3/2} \norm{\Liederivative{\Reebfield}g} \).

\begin{lemma}
	Let \( g \) be a critical metric.
	If \( (\Reebfield, u_{+}, u_{-} = \phi u_{+}) \) is a local orthonormal frame of eigenvector fields of \( h \) (i.e., \( h u_{\pm} = \pm \mu u_{\pm} \)) with $\mu = 2^{-3/2} \norm{\Liederivative{\Reebfield}g} > 0$, then
	\begin{equation}
		[\Reebfield, u_{\pm}]
		=
		\mu
		u_{\mp}
		\label{eq:upm-Lie-brackets}
		.
	\end{equation}
	\label{thm:upm-Lie-brackets}
	
\end{lemma}

\begin{proof}
	For the Lie brackets, using the last identity of \cref{eq:LRphi-properties}, we have that
	\begin{equation*}
		\nabla_{u_{\pm}} \Reebfield
		=
		h\phi u_{\pm}
		=
		\pm h u_{\mp}
		=
		-
		\mu u_{\mp}
		.
	\end{equation*}
	Since \( g \) is critical, \( \nabla_{\Reebfield}h = 0 \) by \cref{thm:critical-g-equivalent-to-nablaRh=0}, so
	\begin{equation*}
		h(\nabla_{\Reebfield} u_{\pm})
		=
		\nabla_{\Reebfield}(h u_{\pm})
		=
		\pm \mu \nabla_{\Reebfield} u_{\pm}
		,
	\end{equation*}
	since \( \nabla_{\Reebfield}\mu = 0 \) by \cref{thm:scalar-torsion-is-first-integral-for-Reeb-field}.
	The \( \pm \mu \)-eigenspace of \( h \) is spanned by \( u_{\pm} \), but \( g(u_{\pm}, \nabla_{\Reebfield} u_{\pm}) = 0 \), so we conclude that \( \nabla_{\Reebfield} u_{\pm} = 0 \).
	This allows us to calculate
	\begin{equation}
		[\Reebfield, u_{\pm}]
		=
		\nabla_{\Reebfield} u_{\pm}
		-
		\nabla_{u_{\pm}} \Reebfield
		=
		0
		-
		(-\mu u_{\mp})
		=
		\mu u_{\mp}
		,
		\label{eq:Lie-bracket-Reeb-h-eigenvectors}
	\end{equation}
	which is what we wanted to show.
\end{proof}

\Cref{thm:upm-Lie-brackets} has a remarkable implication on the dynamics of the Reeb field of a cosymplectic structure in the regions where \( \norm{\Liederivative{\Reebfield}g}^2 \) is positive.
For the statement, we need to recall the definition of a hyperbolic set.

\begin{definition}
	Let \( X \) be a vector field and \( K \) a compact set invariant under the flow \( \Phi^t \) of \( X \).
	Then \( K \) is called hyperbolic for \( X \) if there exists a constant \( c > 0 \), a number \( 0 < \lambda < 1 \), a metric (on a neighbourhood of \( K \)), as well as a \( \Phi^{t} \)-invariant splitting (not necessarily smooth) \( TM = \langle X \rangle \oplus E_+ \oplus E_- \) such that
	\begin{equation*}
		\norm{\Phi^{\mp t}_* v}^2
		\leq
		c \lambda^{t}
		\norm{v}^2
	\end{equation*}
	for all \( t \) whenever \( v \in E_\pm \).
	The bundles \( E_{+} \) and \( E_{-} \) are respectively called the strong unstable and strong stable bundles of the flow.

	  Additionally, it is said that a flow on a compact manifold is \emph{Anosov} if the entire manifold is a hyperbolic set for the flow.
	In the same way, a vector field is Anosov if its flow is.
	Lastly, a diffeomorphism \( f \) of \( M \) is Anosov if there is an \( f \)-invariant splitting \( TM = E_{+} \oplus E_{-} \) satisfying the above condition with \( t \) replaced by \( n \in \mathbb{Z} \).
	\label{def:hyperbolic-set-and-Anosov}
\end{definition}
\begin{lemma}
	Let \( (M, \alpha, \beta) \) be a compact cosymplectic manifold with Reeb vector field \( \Reebfield \) equipped with a critical compatible metric \( g \).
	Let \( K \subseteq M \) be a compact set that is invariant under the flow of $R$ and assume that on \( K \) the torsion is nowhere vanishing, i.e., \( \norm{\Liederivative{\Reebfield}g}^2 > 0 \).
	Then \( K \) is a hyperbolic set for \( \Reebfield \) and the hyperbolic splitting is by smooth line bundles.
	\label{thm:critical-metric-with-positive-torsion-is-hyperbolic}
\end{lemma}

\begin{proof}
	The torsion is assumed positive on \( K \), and thus on an open neighbourhood of \( K \), so around any point of \( K \) a local orthonormal frame \( (\Reebfield, u_{\pm}) \) of eigenvectors of \( h \) exists, see \cref{thm:LRphi-properties}.
	Orienting it by choosing \( u_{-} = \phi u_{+} \), the local frame is unique up to replacing \( (u_{+}, u_{-}) \) by \( (-u_{+}, -u_{-}) \).
	Now do a basis change to \( v_{\pm} = u_{+} \pm u_{-} \).
	The uniqueness up to sign of the \( u_{\pm} \) implies that the line bundles \( E_{\pm} \) locally spanned by the \( v_{\pm} \) are global, this gives a global splitting of the manifold which is locally given as \( \langle \Reebfield \rangle \oplus \langle v_{+} \rangle \oplus \langle v_{-} \rangle \).
	The \( v_{\pm} \) are smooth since the eigenvalues \( \pm \mu = \pm 2^{-3/2}\norm{\Liederivative{\Reebfield}g} \) of \( h \) are distinct, so the splitting is also.
	
	Let \( \Phi^{t} \) denote the flow of the Reeb vector field, let \( p \in K \), and let \( t \) be small enough that \( \Phi^{t}(p) \) still lies in the neighbourhood of \( p \) where the local vector fields \( v_{\pm} \) are defined.
	By \cref{thm:upm-Lie-brackets}, \( [\Reebfield, v_{\pm}] = \pm \mu v_{\pm} \), so \( \Phi^{t}_{*} (v_{\pm})_p = e^{\mp \mu t}(v_{\pm})_{\Phi^{t}(p)} \) since \( \mu \) is a first integral of the Reeb flow by \cref{thm:scalar-torsion-is-first-integral-for-Reeb-field}.
	This both implies that \( \norm{\Phi^{t}_{*} v_{\pm}} = e^{\mp \mu t} \norm{v_{\pm}} = e^{\mp \mu t} \) for such small enough \( t \) and that the \( E_{\pm} \) are invariant line bundles.

	Since \( K \) is compact there is some \( \epsilon > 0 \) such that for all \( \abs{t} < \epsilon \) and all \( v \in E_{\pm} \), \( \norm{\Phi^{t}_{*} v} = e^{\mp \mu t} \norm{v} \).
	Now let \( p \in K \), \( v \in (E_{\pm})_{p} \) and \( t > 0 \) be given.
	Then choose an \( n \) sufficiently large to have \( \abs{t/n} < \epsilon \) to calculate that
	\begin{equation*}
		\norm{
			\Phi^{t}_{*} v
		}
		=
		\norm{
			\Phi^{t/n}_{*}\big( 
				\Phi^{(n-1)t/n}_{*}(v)
			\big)
		}
		=
		e^{\mp t/n}
		\norm{
			\Phi^{(n-1)t/n}_{*}(v)
		}
		=
		\cdots
		=
		e^{\mp t} \norm{v}
		.
	\end{equation*}
	This shows that \( K \) is a hyperbolic set for the Reeb flow.
\end{proof}

The following is immediate from \cref{thm:critical-metric-with-positive-torsion-is-hyperbolic}.

\begin{corollary}
	Let \( (M, \alpha, \beta) \) be a compact cosymplectic manifold equipped with a compatible metric \( g \) with nowhere vanishing torsion.
	Then the Reeb vector field is Anosov with a smooth hyperbolic splitting.
	The stable direction is the eigendirection of the operator \( h\phi \) corresponding to the positive eigenvalue \( \mu = 2^{-3/2}\norm{\Liederivative{\Reebfield}g} \), the unstable direction the one corresponding to \( -\mu \).
	\label{thm:stable-unstable-directions-as-kernels-of-global-tensors}
\end{corollary}

Using \cref{thm:critical-metric-with-positive-torsion-is-hyperbolic} we can also prove that for a critical metric, \( \norm{\Liederivative{\Reebfield}g}^2 \) is in fact a constant.
This was proven for contact manifolds as part of the classification of critical metrics~\cite{hozooriAnosovContactMetrics2023, mitsumatsuExistenceCriticalCompatible2025}.

\begin{lemma}
	Let \( (M, \alpha, \beta) \) be a compact cosymplectic manifold equipped with a compatible metric \( g \) which is critical for the Chern--Hamilton functional.
	Then the torsion \( \norm{\Liederivative{\Reebfield}g}^2 \) is constant.
	\label{thm:Lagrangian-density-is-constant-for-critical-metrics}
\end{lemma}

\begin{proof}
If the torsion vanishes everywhere, the result is trivial.
	Otherwise, by Sard's theorem, it will have a (positive) regular value \( c \).
	The preimage of this value will be a \( 2 \)-dimensional compact submanifold, which we can assume to be connected without any loss of generality (otherwise, take a connected component).

	Additionally, since $M$ is compact, the regular values form an open set of $\mathbb{R}$, so \( K:= \left(\norm{\Liederivative{\Reebfield}g}^2\right)^{-1}( [c-\epsilon, c+\epsilon] ) \) is a compact smooth submanifold with boundary on which the torsion is strictly positive, for $\epsilon > 0$ small enough.
	Since the torsion is a first integral of the Reeb flow, each level set is invariant under the Reeb flow. In particular, $K$ is a compact invariant set.
	By \cref{thm:critical-metric-with-positive-torsion-is-hyperbolic} \( K \) is hyperbolic for the Reeb field.
	In~\cite{aguilarNonexistenceCodimensionOne2005} it is shown that codimension-one Anosov flows (i.e., flows where either the stable or unstable direction is \(1\)-dimensional—in our case, both are) on compact manifolds cannot have invariant surfaces.
	Therefore, we conclude that the torsion must be constant on the whole manifold.
\end{proof}

The above \namecref{thm:Lagrangian-density-is-constant-for-critical-metrics} shows that there are only two types of Reeb vector fields on manifolds equipped with critical metrics: either the Reeb vector field is Killing or \( \norm{\Liederivative{\Reebfield}g}^2 \) is some positive constant and the Reeb vector field is Anosov.
To finish the proof of \cref{thm:classification} we are thus only missing to describe the manifolds admitting critical metrics with (constant) positive torsion.

\section{Proof of the main theorem}
\label{sec:proof}
In this \namecref{sec:proof} we shall complete the proof of \cref{thm:classification}.
In particular, we assume that all manifolds are \( 3 \)-dimensional.
First, observe that in this case, a co-Kähler structure is the same as a compatible metric with \( \Liederivative{\Reebfield}g = 0 \)~\cite{goldbergIntegrabilityAlmostCosymplectic1969}.
The energy functional is non-negative, so such a metric minimizes the functional and is therefore also critical.

On the other hand, according to \cref{thm:Lagrangian-density-is-constant-for-critical-metrics}, any critical metric has constant torsion, so any other critical metric will have constant positive torsion.
Such metrics will be classified in \cref{sec:critical-metrics-Anosov}, and in \cref{sec:positive-torsion-critical-metric-minimizes-energy} we show that they also minimize the functional.
This will be done by expressing the energy of arbitrary compatible metrics in terms of the energy of critical ones.
These steps cover all the claims made in \cref{thm:classification}.

\subsection{Critical metrics with non-vanishing torsion}
\label{sec:critical-metrics-Anosov}

This \namecref{sec:critical-metrics-Anosov} describes the structure of \( 3 \)-dimensional, compact cosymplectic manifolds equipped with critical metrics with constant, positive torsion \( \norm{\Liederivative{\Reebfield}g}^2 > 0 \).
By \cref{thm:LRphi-properties,thm:Lagrangian-density-is-constant-for-critical-metrics}, \( h = (1/2) \Liederivative{\Reebfield}\phi \) is then orthogonally diagonalizable at all points with constant eigenvalues \( 0 \) and \( \pm \mu \), where \( \mu = 2^{-3/2} \norm{\Liederivative{\Reebfield}g} \).
The eigendirections are uniquely defined due to the three eigenvalues being distinct, so in a neighbourhood of every point \( h \) there is a local orthonormal eigenframe \( (R, u_{+}, u_{-} = \phi u_{+}) \).
Let us first investigate the structure of the vector fields \( v_{\pm} := u_{+} \pm u_{-} \).

\begin{lemma}\label{L:brackets}
	Let \( g \) be a critical metric with positive torsion.
	Then,
	\begin{equation*}
		[\Reebfield, v_{\pm}]
		=
		\pm \mu v_{\pm}
		,\quad\text{and}\quad
		[v_{+}, v_{-}] = 0
		,
	\end{equation*}
	where \( \mu = 2^{-3/2} \norm{\Liederivative{\Reebfield}g} \).
	\label{thm:vpm-Lie-brackets}
\end{lemma}

\begin{proof}
	That \( [\Reebfield, v_{\pm}] = \pm \mu v_{\pm} \) follows from \cref{thm:upm-Lie-brackets}.
	Since the \( u_{\pm} \) are orthogonal to \( \Reebfield \) which is the metric dual of \( \alpha \) (by \cref{def:compatible-metric}), the \( u_{\pm} \) lie in the integrable foliation \( \ker\alpha \).
	In particular, we can write \( [v_{+}, v_{-}] = a_{+} v_{+} + a_{-} v_{-} \) for smooth (locally defined) functions \( a_{\pm} \).
	Letting \( \Phi^{t} \) denote the flow of the Reeb vector field, the identity \( [\Reebfield, v_{\pm}] = \pm \mu v_{\pm} \) implies that \( \Phi^{t}_{*} v_{\pm} = e^{\mp \mu t} v_{\pm} \).
	Using this, we compute
	\begin{equation*}
		a_{+}v_{+}
		+
		a_{-}v_{-}
		=
		[e^{-\mu t} v_{+}, e^{\mu t} v_{-}]
		=
		\Phi^{t}_{*}[v_{+}, v_{-}]
		=
		(a_{+} \circ \Phi^{-t}) e^{-\mu t} v_{+}
		+
		(a_{-} \circ \Phi^{-t}) e^{\mu t} v_{-}
		.
	\end{equation*}
	Since only the eigendirections of \( h \), not its eigenvectors, are uniquely defined at every point, the \( v_{\pm} \) are not uniquely defined at every point.
	The above calculation is thus only valid locally.

	However, the only freedom for the \( v_{\pm} \) is changing \( (v_{+}, v_{-}) \) out for \( (-v_{+}, -v_{-}) \) (since \( v_{-} \) is determined by \( v_{+} \) as \( v_{-} = -\phi v_{+} \)).
	This changes the functions \( a_{\pm} \) to \( -a_{\pm} \), in particular, the functions \( \abs{a_{\pm}} \) \emph{are} globally defined.
	Covering paths \( \Phi^{t}(p) \) by neighbourhoods with locally well-defined \( v_{\pm} \) vector fields, we have for all \( t \in \mathbb{R} \)  that
	\begin{equation*}
		\abs{a_{\pm} \circ \Phi^{t}}
		=
		\abs{a_{\pm}} e^{\mp \mu t}
		,
	\end{equation*}
	which, since \( \mu = 2^{-3/2}\norm{\Liederivative{\Reebfield}g} \) is constant and non-vanishing and the \( \abs{a_{\pm}} \) are continuous functions on a compact manifold, is impossible unless \( a_{\pm} = 0 \) everywhere.
\end{proof}

By rescaling \( R \) to \( Y = \mu^{-1}R \) we get the standard basis for the Lie algebra \( \modelLieAlgebra \) (see \cref{sec:algebraic-anosov}).
The existence of such a local frame is in fact equivalent to the metric being critical.

\begin{lemma}
	Let \( (M, \alpha, \beta) \) be a compact cosymplectic manifold equipped with a compatible metric \( g \), associated \( (1, 1) \)-tensor field \( \phi \), and Reeb vector field \( \Reebfield \).
	Then \( g \) is a critical metric with positive torsion if and only if there exists, in a neighbourhood of any point,  a local orthonormal frame \( (\Reebfield, v_{+}, v_{-} = -\phi v_{+}) \) such that \( [v_{+}, v_{-}] = 0 \) and \( [\Reebfield, v_{\pm}] = \pm \mu v_{\pm} \) for some constant \( \mu > 0 \).
	In the affirmative case, \( \norm{\Liederivative{\Reebfield}g}^2 = 8\mu^2 \).
	\label{thm:critical-metric-criterion-by-local-frame}
\end{lemma}

\begin{proof}
	\Cref{thm:vpm-Lie-brackets} provides one direction.
	For the other, we will show that the assumptions imply that \( \nabla_{\Reebfield}h \) vanishes, which by \cref{thm:critical-g-equivalent-to-nablaRh=0} is equivalent to the metric being critical.
	Since \( (\nabla_{\Reebfield}h)(\Reebfield) = 0 \) we just need to calculate
	\begin{equation*}
		(\nabla_{\Reebfield}h)(v_{\pm})
		=
		\nabla_{\Reebfield}(h v_{\pm})
		-
		h \nabla_{\Reebfield} v_{\pm}
		.
	\end{equation*}
	Using just the Lie brackets \( [\Reebfield, v_{\pm}] \) and that \( \phi v_{\pm} = \mp v_{\mp} \) we have that
	\begin{equation}
		2hv_{\pm}
		=
		(\Liederivative{\Reebfield}\phi)(v_{\pm})
		=
		[\Reebfield, \phi v_{\pm}]
		-
		\phi[\Reebfield, v_{\pm}]
		=
		\mp [\Reebfield, v_{\mp}]
		-
		\phi(\pm \mu v_{\pm})
		=
		2\mu v_{\mp}
		.
		\label{eq:hvpm=muvpm-for-local-frame}
	\end{equation}
	Additionally, by \cref{thm:LRphi-properties}, we have that
	\begin{equation*}
		\nabla_{\Reebfield} v_{\pm}
		=
		[\Reebfield, v_{\pm}]
		+
		\nabla_{v_{\pm}} \Reebfield
		=
		\pm \mu v_{\pm}
		+
		h\phi v_{\pm}
		=
		\pm \mu v_{\pm}
		+
		(\mp \mu v_{\pm}e)
		=
		0
        .
	\end{equation*}
	This shows that $(\nabla_{\Reebfield}h)(v_{\pm}) = \mu\nabla_{\Reebfield} v_{\mp} - h(0) = 0$,
	and therefore \( \nabla_{\Reebfield} h = 0 \).
	\Cref{eq:hvpm=muvpm-for-local-frame} implies that \( v_{+} + v_{-}\) is an eigenvector of \( h \) with eigenvalue \( \mu \), so \cref{thm:LRphi-properties} let us conclude that \( \norm{\Liederivative{\Reebfield}g}^2 = 8\mu^{2} \).
\end{proof}

We can extract another observation from this \namecref{thm:critical-metric-criterion-by-local-frame}.

\begin{corollary}
	The compatible metrics constructed in \cref{ex:critical-metric-Sol-quotient} are critical.
	\label{thm:Sol-model-metrics-are-critical}
\end{corollary}

The final piece in classifying the manifolds admitting critical metrics with positive torsion is to show that they are all the result of this construction.
The strategy is, given a manifold \( (M, \alpha, \beta, g) \) with a critical metric, to use the local \( \modelLieAlgebra \) structure given by \cref{thm:vpm-Lie-brackets} to reconstruct \( \modelLieGroup \) on the universal cover of \( M \).

\begin{proposition}
	Let \( (M, \alpha, \beta) \) be a compact, connected cosymplectic manifold equipped with a critical compatible metric \( g \) with positive torsion \( \norm{\Liederivative{\Reebfield}g}^2 > 0\).
	Then \( \modelLieGroup \) with its cosymplectic and Riemannian structure from \cref{ex:compatible-metric-Sol} with \( \mu = 2^{-3/2} \norm{\Liederivative{\Reebfield}g} \) covers (locally cosymplectically and isometrically) \( M \).
	\label{thm:positive-torsion-implies-homogeneous-space}
\end{proposition}

\begin{proof}
	Since the torsion is positive (and constant), \( h \) is diagonalizable at every point with distinct eigenvalues (\cref{thm:LRphi-properties}), so the eigendirections of \( h \) define line bundles on \( M \) spanning the tangent space at every point.
	Now lift the cosymplectic and metric structure to the universal cover \( \tilde{M} \) of \( M \).
	We will denote the lifted tensor and vector fields with the same symbols as the corresponding objects on \( M \).
	On \( \tilde{M} \), the (lift of the) eigendirections are oriented line bundles and thus spanned by a global frame of vector fields \( (\Reebfield, v_{+}, v_{-} = -\phi v_{+}) \) satisfying the Lie bracket relations \( [\mu^{-1}\Reebfield, v_{\pm}] = \pm v_{\pm} \) and \( [v_{+}, v_{-}] = 0 \), cf. \cref{L:brackets}.
	This defines a Lie algebra action of \( \modelLieAlgebra \) (see \cref{sec:algebraic-anosov}) on \( \tilde{M} \).
	The vector fields \( \Reebfield \) and \( v_{\pm} \) are complete, so the action integrates to a right action of \( \modelLieGroup \) on \( \tilde{M} \).

	This action is transitive.
	Notice first that for any fixed \( p \in \tilde{M} \), the map \( \modelLieGroup \to \tilde{M} \), \( g \mapsto g.p \) is a local diffeomorphism since its derivative \( \modelLieAlgebra \to T_p \tilde{M} \) is an isomorphism.
	This implies that all orbits of the action \( \modelLieGroup \actsOn \tilde{M} \) are open and thus also closed, since the complement of an orbit is the union of all other orbits.
	Since \( \tilde{M} \) is connected, the action only has one orbit.

	This implies that \( \tilde{M} \) is a homogeneous space of \( \modelLieGroup \), and since \( \tilde{M} \) is simply connected we in fact have \( \tilde{M} = \modelLieGroup \).
	Under this correspondence, the (right) action of \( \modelLieGroup \) on \( \tilde{M} \) is by right-multiplication.
	Recall the basis \( (Y, X_{\pm}) \) of \( \modelLieAlgebra \) (\cref{sec:algebraic-anosov}).
	That the action \( \modelLieGroup \actsOn \tilde{M} = \modelLieGroup \) integrates the Lie algebra action \( \modelLieAlgebra \actsOn \tilde{M} \) means that \( Y^{\sharp} = \mu^{-1} \Reebfield \) and \( X_{\pm}^{\sharp} = v_{\pm} \).
	Here \( Y^{\sharp} \) denotes the fundamental vector field of \( Y \) with regards to the action (and similarly for \( X_{\pm}^{\sharp} \)), defined at \( g \in \modelLieGroup = \tilde{M} \) as
	\begin{equation*}
		Y^{\sharp}_{g}
		=
		\odfrac{}{t}\Big\vert_{t=0}
		g e^{tY}
		.
	\end{equation*}
	An elementary fact is that fundamental vector fields of actions by right-multiplication are invariant under left-multiplication.
	In particular, \( \Reebfield \) and \( v_{\pm} \) are left-invariant vector fields on \( \modelLieGroup = \tilde{M} \).
	This means that the lift of the cosymplectic and metric structures on \( M \) to \( \tilde{M} \) are exactly the ones considered in \cref{ex:compatible-metric-Sol}, or, in other words, \( M \) is covered by this model.
\end{proof}

Recall the algebraic Anosov flows from \cref{sec:algebraic-anosov} (and that the Reeb flow of a cosymplectic manifold with a critical metric is Anosov, \cref{thm:stable-unstable-directions-as-kernels-of-global-tensors}).
The following is then immediate from the above result with the reparametrization being by scaling by the parameter \( \mu \).

\begin{corollary}
	The Reeb flow of a compact, connected cosymplectic manifold with a positive torsion critical metric is, up to reparametrization, an algebraic Anosov flow of \( \modelLieGroup \) type.
	\label{thm:critical-metric-Reeb-is-algebraic-Anosov}
\end{corollary}

\subsection{Equivalence of models for critical metrics}
\label{sec:Sol-and-mapping-torus-models-equivalent}
The combination of \cref{thm:Sol-model-metrics-are-critical} and \cref{thm:positive-torsion-implies-homogeneous-space} completely describes the critical metrics with positive torsion as being the ones constructed in \cref{ex:critical-metric-Sol-quotient}.
In \cref{thm:classification} this classification was instead phrased in terms of mapping tori.
In this \namecref{sec:Sol-and-mapping-torus-models-equivalent} we will prove the equivalence of these two models of critical metrics, \cref{ex:critical-metric-mapping-torus} and \cref{ex:critical-metric-Sol-quotient}.

We prove that the two models are cosymplectomorphic and isometric.
Since we know that the metrics on quotients of \( \modelLieGroup \) are critical by \cref{thm:Sol-model-metrics-are-critical} this also shows that the metrics on mapping tori from \cref{ex:critical-metric-mapping-torus} are critical.

\begin{lemma}
	Let \( (M, \alpha, \beta, g) \) be a compact cosymplectic manifold with a compatible metric and assume that \( M \) is covered by \( \modelLieGroup \) equipped with the structures from \cref{ex:critical-metric-Sol-quotient} (the covering map being locally cosymplectic and isometric) for some \( \mu \neq 0 \).
	Then \( M \) is cosymplectomorphic and isometric to a mapping torus of \( \mathbb{T}^{2} \) by a hyperbolic toral automorphism with the structures defined in \cref{ex:critical-metric-mapping-torus} for some \( \dtFactor \) and \( V \).

	Conversely, any such cosymplectic mapping torus is covered by \( \modelLieGroup \) with the cosymplectic and metric structure of \cref{ex:critical-metric-Sol-quotient}.
	\label{thm:Sol-covered=mapping-torus}
\end{lemma}

\begin{proof}
	Given a mapping torus, observe that the orthonormal vector fields \( (R, v_{\pm}) \) satisfy the Lie bracket criteria of \cref{thm:critical-metric-criterion-by-local-frame}.
	\Cref{thm:positive-torsion-implies-homogeneous-space} then implies that the mapping torus is covered by the structure on \( \modelLieGroup \) from \cref{ex:critical-metric-Sol-quotient}.

	For the other direction consider a covering \( \modelLieGroup \to M \) which is locally both cosymplectic and isometric.
	This covering is a quotient map \( \modelLieGroup \to \Gamma\backslash\modelLieGroup \) where the deck transformation group \( \Gamma \) is a subgroup of the group \( \Aut(\modelLieGroup) \) of isometric cosymplectomorphisms.
	In \cite{scottGeometries3Manifolds1983} it is shown that the isometry group of \( \modelLieGroup \) is generated by \( \modelLieGroup \) acting on itself by left-multiplication and the eight diffeomorphisms of the form \( (t, x_{+}, x_{-}) \mapsto (t, \pm x_{+}, \pm x_{-}) \) and \( (t, x_{+}, x_{-}) \mapsto (-t, \pm x_{-}, \pm x_{+}) \).
	Note that this is the isometry group of \( \modelLieGroup \) when equipped with the metric defined in \cref{ex:compatible-metric-Sol} with \( \mu = 1 \), but it is clear now that the isometry group in fact is independent of the choice of \( \mu \).

	The cosymplectic structure was constructed left-invariant, so it is preserved by the action of \( \modelLieGroup \).
	Among the extra diffeomorphisms, only \( r \colon (t, x_{+}, x_{-}) \mapsto (t, -x_{+}, -x_{-}) \) preserves the cosymplectic structure on \( \modelLieGroup \).
	The full group \( \Aut(\modelLieGroup) \) is thus generated by this flip and \( \modelLieGroup \) acting on itself by left-multiplication.
	We can think of \( \langle r \rangle \) as \( \mathbb{Z}_{2} = \{ -1, 1 \} \) with \( r \) corresponding to \( -1 \).
	Denote by \( \psi \) the action of \( \langle r \rangle \simeq \mathbb{Z}_{2} \actsOn \modelLieGroup \) given by \( \psi(r)(t, v) = (t, -v) \).
	Note that if \( \epsilon \in \langle r \rangle \) and \( g \in \modelLieGroup \), \( \epsilon \circ g = \psi(\epsilon)(g) \circ \epsilon \) as elements of \( \Aut(\modelLieGroup) \).
	We conclude that \( \Aut(\modelLieGroup) \) is isomorphic to the semidirect product \( \mathbb{Z}_{2} \ltimes_{\psi} \modelLieGroup \) where an element \( (\epsilon, g) \) acts on point \( h \in \modelLieGroup \) as \( (\epsilon, g) h = g( \psi(\epsilon)(h) ) \).
	
	Now let \( \Gamma \) be a subgroup of \( \Aut(\modelLieGroup) \) for which \( \Gamma\backslash\modelLieGroup \) is compact and the quotient map is a covering.
	Let \( \Gamma_0 \) denote the normal subgroup of \( \Gamma \) of elements \( (+1, g) \) obtained as the kernel of the homomorphism \( \Gamma \to \mathbb{Z}_{2} \) sending \( (\epsilon, g) \mapsto \epsilon \).

	There are two possibilities.
	If \( \Gamma = \Gamma_0 \), \( \Gamma \) is just a cocompact, discrete subgroup of \( \modelLieGroup \).
	These are classified in \cite{auslanderFlowsHomogeneousSpaces1963}: \( \Gamma_0 \) is generated by three elements of \( \modelLieGroup \), \( (k, 0) \) (let us take \( k > 0 \)), \( (0, u) \), and \( (0, v) \), satisfying that \( e^{k} + e^{-k} \) is an integer, that \( u \) and \( v \) are linearly independent vectors, and that the diagonal matrix with diagonal \( (e^{k}, e^{-k}) \) preserves the lattice \( \mathbb{Z}u + \mathbb{Z} v\).
	This last requirement is equivalent to the matrix
	\begin{equation*}
		L
		:=
		\begin{pmatrix}
			u & v
		\end{pmatrix}^{-1}
		\begin{pmatrix}
			e^{k} & 0\\
			0 & e^{-k}
		\end{pmatrix}
		\begin{pmatrix}
			u & v
		\end{pmatrix}
	\end{equation*}
	having integer entries.
	Changing to the \( (u, v) \) basis and rescaling \( \mathbb{R} \) by \( 1/k \), we can present \( \modelLieGroup \) as \( \mathbb{R} \ltimes \mathbb{R}^2 \) with the action of \( \mathbb{R} \) being \( t.v = L^t v \) and the subgroup \( \Gamma_0 \) being \( \mathbb{Z} \ltimes \mathbb{Z}^{2} \).
	Now it is direct to see that \( \Gamma_0\backslash\modelLieGroup \) is a mapping torus of \( \mathbb{T}^2 \) by the hyperbolic toral automorphism \( L^{-1} \).
	Note that under this coordinate change, the vectors \( \partial_{x_{\pm}} \) go to the eigenvectors of \( L \) with eigenvalues \( e^{\pm k} \).
	We see that this gives an isometric cosymplectomorphism to a model from \cref{ex:critical-metric-mapping-torus} (with \( \lambda = e^{k} \), \( V = \abs{u \times v} \) the area of fundamental domain of the torus measured using the \( 2 \)-form on \( \modelLieGroup \), and \( \tau = \mu/k \)).
	
	The other case is that \( \Gamma \) is not equal to \( \Gamma_0 \).
	Then the map \( \Gamma \to \mathbb{Z}_{2} \) from before is surjective, so \( \Gamma_{0} \) has index \( 2 \) in \( \Gamma \), which means that the map \( \Gamma_{0} \backslash \modelLieGroup \to \Gamma \backslash \modelLieGroup \) is a double cover.
	In particular, \( \Gamma_{0} \backslash \modelLieGroup \) is also compact, so \( \Gamma_{0} \) is of the form described above and we can change coordinates such that \( \Gamma_{0} \) becomes isomorphic to \( \{ 1 \} \ltimes \left(\mathbb{Z} \ltimes \mathbb{Z}^{2}\right) \).
	By multiplying with elements from \( \Gamma_{0} \) it can be seen that \( \Gamma \setminus \Gamma_{0} \) exactly consists of elements \( (-1, t_0 + k, L^{k} v_0 + u) \) where \( k \in \mathbb{Z} \), \( u \in \mathbb{Z}^2 \), and \( (-1, t_{0}, v_{0}) \in \Gamma \) is some fixed element.
	The square of this element is \( (1, 2t_{0}, (1-L^{t})v_{0}) \in \Gamma_{0} \), in particular \( 2t_{0} \in \mathbb{Z} \).

	Recall that the quotient map \( \modelLieGroup \to \Gamma \backslash \modelLieGroup \) is a covering map, in particular the action \( \Gamma \actsOn \modelLieGroup \) is free.
	Thus we cannot have \( t_{0} \in \mathbb{Z} \) since in that case there would be an element \( (-1, 0, v) \in \Gamma \) which fixes the points \( (t, (1/2) v_{0}) \in \modelLieGroup \).
	We conclude that we can take \( t_{0} = 1/2 \).

	From this explicit description of \( \Gamma \) we see as before that \( \Gamma \backslash \modelLieGroup \) is a mapping torus of \( \mathbb{T}^{2} = \mathbb{R}^{2}/\mathbb{Z}^{2} \), but this time glued at \( 1/2 \) by the diffeomorphism \( p \mapsto -L^{-1/2}p + L^{-1}v_{0} \).
	By doing another coordinate change \( (t, v) \mapsto (t/2, v + (L - 1)^{-1} L^{-1}v_{0}) \) on \( \modelLieGroup \), the mapping torus gets glued at \( 1 \) by the gluing map \( -L^{-1/2} \).
	This coordinate change preserves the (local) vector fields \( \partial_{x_{\pm}} \) on \( \modelLieGroup \) and thus, as for the case \( \Gamma = \Gamma_{0} \), the correspondence from the \( \modelLieGroup \) quotient to a mapping torus sends \( \partial_{x_{\pm}} \) to the eigenvectors of \( -L^{-1/2} \).
	We see then that, as in that situation, the cosymplectic and metric structures on \( \Gamma \backslash \modelLieGroup \) are of the type from \cref{ex:critical-metric-mapping-torus} on the mapping torus.
	Notice that this time \( \tau = \mu/2k \) due to the extra rescaling of the \( t \)-coordinate.
\end{proof}

As mentioned in the introduction of this \namecref{sec:Sol-and-mapping-torus-models-equivalent}, the manifolds admitting critical metrics with positive torsion are exactly the ones from \cref{ex:critical-metric-Sol-quotient}.
Combining this with the equivalence of the \( \modelLieGroup \) and mapping torus models we have just shown, we can summarize what we have proven so far as follows.

\begin{corollary}
	The constructions in \cref{ex:critical-metric-Sol-quotient}, and, equivalently, in \cref{ex:critical-metric-mapping-torus}, give all connected cosymplectic manifolds with critical compatible metrics with positive torsion.
	\label{thm:critical-metric-iff-Sol-quotient}
\end{corollary}

Note that in the correspondence of \cref{thm:Sol-covered=mapping-torus}, the parameter \( \mu \) from the \( \modelLieGroup \) model does not exactly correspond to the parameter \( \dtFactor \) from the mapping torus model.
The difference is due to a rescaling of \( \mathbb{R} \) by either \( k^{-1} \) or \( (2k)^{-1} \).

The case \( \Gamma_{0} \neq \Gamma \) corresponds to the line bundles spanned by the local vector fields \( v_{\pm} \) (in the mapping torus viewpoint) and \( X_{\pm} \) (in the \( \modelLieGroup \) quotient viewpoint) not being orientable and the gluing diffeomorphism \( L \) of the mapping torus having negative eigenvalues.
Conversely, the manifold is a Lie group quotient of \( \modelLieGroup \) (i.e., \( \Gamma_{0} = \Gamma \)) if and only if the gluing map has positive eigenvalues or equivalently that the \( v_{\pm} \)-directions are orientable.

If the gluing map has negative eigenvalues (such that the \( v_{\pm} \)-directions are non-orientable), the mapping torus glued by \( L^{2} \) is a double cover of the manifold where these line bundles become orientable.
From the \( \modelLieGroup \) point of view, a manifold \( M = \Gamma\backslash\modelLieGroup \) with \( \Gamma_0 \neq \Gamma \) has a double cover \( \widehat{M} = \Gamma_0\backslash\modelLieGroup \) given by the partial quotient by the discrete subgroup \( \Gamma_0 \) of \(\modelLieGroup\) acting by left-multiplication.

\begin{remark}
	Recall from \cref{thm:critical-metric-Reeb-is-algebraic-Anosov} that the Reeb flow on a cosymplectic manifold with a critical metric is algebraic Anosov of \( \modelLieGroup \) type (up to reparametrization).
	The converse, however, is not true.
	Consider, for example, the mapping torus model from \cref{ex:critical-metric-mapping-torus} but glued by a matrix \( L \in \Liegroup{GL}(2, \mathbb{Z}) \) with determinant \( -1 \) and no eigenvalue of absolute value \( 1 \).
	Since \( L^{2} \in \Liegroup{SL}(2, \mathbb{Z}) \), the suspension flow on \( \mathbb{T}^{2}_{L^{2}} \) is algebraic Anosov by \cref{thm:Sol-covered=mapping-torus} (up to reparametrization) and is a double cover of the flow on \( \mathbb{T}^{2}_{L} \), so also the suspension flow there is algebraic Anosov of \( \modelLieGroup \) type (up to reparametrization).

	However, since \( L \) is orientation reversing, the mapping torus \( \mathbb{T}^{2}_{L} \) is not orientable and in particular cannot be a cosymplectic manifold.
	\label{rem:algebraic-Anosov-does-not-imply-Reeb-with-critical-metric}
\end{remark}

\subsection{Cosymplectic manifolds without critical metrics}
\label{sec:no-critical-metrics-construction}

With \cref{thm:critical-metric-iff-Sol-quotient} we have proven the classification part of \cref{thm:classification}: that a cosymplectic manifold admitting a critical metric is either co-Kähler or a mapping torus by a hyperbolic toral automorphism.
In this \namecref{sec:no-critical-metrics-construction} we include constructions of cosymplectic manifolds not admitting critical metrics for the Chern--Hamilton functional.

It turns out that \emph{formality}, a powerful tool to distinguish symplectic manifolds that admit Kähler structures from those that do not, plays a key role in the construction of cosymplectic manifolds that do not admit critical metrics.
Indeed, formality, which relates a manifold's rational homotopy type to its cohomology algebra, is obstructed by the presence of non-zero Massey products (confer \cite{BFM}).
Since a compact manifold admits a co-Kähler structure if and only if it is the mapping torus of a Hermitian isometry of a Kähler manifold~\cite{liTopologyCosymplecticCoKahler2008},
it follows that all compact co-Kähler manifolds are formal (see e.g., \cite{BFM2}).
This motivated the authors in \cite{BFM} to study formality in the broader context of cosymplectic manifolds, determining examples of cosymplectic manifolds which are not co-Kähler.
In particular, they prove in \cite[Proposition 20]{BFM} that there exist non-formal compact cosymplectic manifolds of dimension~\( m \geq 3 \) with first Betti number \( b_1 \geq 2 \).
Their construction is based on mapping tori of symplectomorphisms of surfaces \( \Sigma_k \) of genus \( k \geq 1 \), where the monodromy acts in a prescribed way on cohomology to ensure the existence of a non-trivial Massey product, establishing non-formality.
The construction in \cite{BFM} yields examples in dimension~\( 3 \) with first Betti number \( b_1 \geq 2 \), and higher-dimensional examples are obtained by taking products with spheres \( (S^2)^\ell \), for \( \ell \geq 0 \).

\begin{proposition}[\cite{BFM}]\label{thm:nonformal-cosymplectic-existence}
	There are non-formal compact manifolds with dimension~\( m \geq 3 \) and \( b_1 \geq 2 \) that admit cosymplectic structures.
	In particular, these manifolds admit no co-Kähler structures.
\end{proposition}

Now recall that the other type of manifolds admitting critical compatible metrics are mapping tori of \( \mathbb{T}^{2} \) glued by a hyperbolic toral automorphism \( L \in \Liegroup{SL}(2, \mathbb{Z}) \).
The diffeomorphism \( L \) induces a map \( L^* \) between the cohomology groups of \( \mathbb{T}^2 \), in particular, \( L^{*} \colon H^{1}(\mathbb{T}^2; \mathbb{R}) \to H^{1}(\mathbb{T}^2; \mathbb{R}) \) simply acts as the matrix \( L \).
Since a hyperbolic toral automorphism \( L \) does not have \( 1 \) as an eigenvalue, \( 1 - L \) is invertible.
This restricts the topology of the resulting mapping torus in the following way.

\begin{lemma}
	Let \( (S, \omega) \) be a connected symplectic surface and \( f \colon S \to S \) a symplectomorphism.
	If \( 1 - f^* \colon H^1(S) \to H^1(S) \) is an isomorphism, then \( H^1(S_f) \simeq H^0(S) \) and \( H^2(S_f) \simeq H^2(S) \).
	\label{thm:mapping-torus-H1}
\end{lemma}

\begin{proof}
	By an application of the Mayer--Vietoris sequence to the mapping torus one obtains a long exact sequence in cohomology, which around \( H^1(S_f) \) looks like
	\begin{equation*}
		\dots
		\longrightarrow
		H^0(S)
		\overset{1-f^*}{\longrightarrow}
		H^0(S)
		\longrightarrow
		H^{1}(S_f)
		\longrightarrow
		H^1(S)
		\overset{1-f^*}{\longrightarrow}
		H^1(S)
		\longrightarrow
		\cdots
	\end{equation*}
	See for example~\cite[example 2.48]{hatcherAlgebraicTopology2001}.
	The assumption that \( 1-f^* \colon H^1(S) \to H^1(S) \) is an isomorphism implies that \( H^1(S_f) \) is a quotient of \( H^0(S) \).
	Since \( S \) is connected, the map \( 1-f^* \colon H^0(S) \to H^0(S) \) is \( 0 \), which implies that \( H^0(S) \rightarrow H^1(S_f) \) is injective, which leads to the claimed isomorphism \( H^0(S) \simeq H^1(S_f) \).
	The long exact sequence around \( H^2(S_f) \) is
	\begin{equation*}
		\dots
		\longrightarrow
		H^1(S)
		\overset{1-f^*}{\longrightarrow}
		H^1(S)
		\longrightarrow
		H^{2}(S_f)
		\longrightarrow
		H^2(S)
		\overset{1-f^*}{\longrightarrow}
		H^2(S)
		\longrightarrow
		\cdots
		.
	\end{equation*}
	The connectedness of \( S \) implies that \( 1-f^* \colon H^2(S) \to H^2(S) \) vanishes, so an argument like the one above gives that \( H^{2}(S_f) \simeq H^{2}(S) \).
\end{proof}

This \namecref{thm:mapping-torus-H1} implies the following result about the Betti numbers.

\begin{corollary}
	A compact connected cosymplectic \(3\)-manifold admitting a critical metric with positive torsion has all Betti numbers equal to \( 1 \).
	\label{thm:critical-metric-positive-torsion-Betti-numbers}
\end{corollary}

\begin{corollary}
	There exist compact \( 3 \)-dimensional manifolds which admit cosymplectic structures, but such that no cosymplectic structure on these manifolds admits a critical compatible metric for the Chern--Hamilton functional.
	\label{thm:manifolds-with-no-critical-metrics-existence}
\end{corollary}

\begin{proof}
	The manifolds of \cref{thm:nonformal-cosymplectic-existence} admit cosymplectic structures but they do not admit any co-Kähler structures.
	Recall that a cosymplectic manifold with a critical metric is either co-Kähler or, by \cref{thm:critical-metric-positive-torsion-Betti-numbers}, has first Betti number \( b_{1} = 1 \).
	Since the manifolds of \cref{thm:nonformal-cosymplectic-existence} have \( b_{1} \geq 2 \) we conclude that they admit no cosymplectic structure admitting a critical metric.
\end{proof}

\begin{remark}
	In \cite{CLM} it is proved that the first Betti number of a co-Kähler manifold is always odd.
	This gives a different way of seeing that the manifolds of \cref{thm:nonformal-cosymplectic-existence} with \( b_{1} \) even cannot admit a co-Kähler structure, and thus also no critical compatible metric.
\end{remark}

\begin{remark}
	In \cite{BFM2}, other conditions such as the hard Lefschetz property are examined as distinguishing features of co-Kähler manifolds.
\end{remark}

\begin{example}
	An explicit example of the constructions in \cite{BFM} of non-formal cosymplectic manifolds is the mapping torus of \( \mathbb{T}^{2} \)  glued by the map
	\begin{equation*}
		L
		=
		\begin{pmatrix}
			1&	1	\\
			0&	1
		\end{pmatrix}
		.
	\end{equation*}
	Since \( \det L = 1 \), the manifold \( \mathbb{T}^{2}_{L} \) admits a suspension cosymplectic structure by equipping the torus with the standard symplectic structure.
	We can directly see that this manifold does not admit any co-Kähler structure.
	Indeed, the long exact sequence used in the proof of \cref{thm:mapping-torus-H1} shows that \( H^{1}(\mathbb{T}^{2}_{L}; \mathbb{R}) \simeq \mathbb{R}^{2} \), but recall that the first Betti number of a co-Kähler manifold is always odd~\cite{CLM}.
\end{example}

\subsection{Critical metrics globally minimize the energy}
\label{sec:positive-torsion-critical-metric-minimizes-energy}

The last claim in \cref{thm:classification} which remains to prove is that critical metrics minimize the energy functional.
In this \namecref{sec:positive-torsion-critical-metric-minimizes-energy} we will present a proof of this fact that is an adaptation of the one in \cite{mitsumatsuExistenceCriticalCompatible2025} to the cosymplectic setting.

Let \( (M, \alpha, \beta) \) be a cosymplectic manifold equipped with a critical metric \( g \), and let \( \tilde{g} \) be any compatible metric.
Our goal is then to show that \( E(\tilde{g}) \geq E(g) \).
As explained in the introduction to this \namecref{sec:proof} we only need to consider the case where the torsion \( \norm{\Liederivative{\Reebfield}g}^2 \) is a positive constant.
In this case there exists by \cref{thm:critical-metric-criterion-by-local-frame} a local \( g \)-orthonormal frame \( (\Reebfield, v_{+}, v_{-}) \) around every point with \( [\Reebfield, v_{\pm}] = \pm \mu v_{\pm} \) and \( [v_{+}, v_{-}] = 0 \), where \( \mu = 2^{-3/2}\norm{\Liederivative{\Reebfield}g} \).

The \( v_{\pm} \) span \( \ker\alpha \) which is \( \tilde g \)-orthogonal to the Reeb vector field, so we can write
\begin{equation*}
	\tilde g
	=
	\alpha \tensorproduct \alpha
	+
	q v_{+}^{\flat} \tensorproduct v_{+}^{\flat}
	+
	r(
		v_{+}^{\flat} \tensorproduct v_{-}^{\flat}
		+
		v_{-}^{\flat} \tensorproduct v_{+}^{\flat}
	)
	+
	p v_{-}^{\flat} \tensorproduct v_{-}^{\flat}
	,
\end{equation*}
where \( \flat \) denotes taking the metric dual of the vector field using the critical metric \( g \).
Moreover, one must have \( p q - r^{2} = 1 \) to have the \( \beta = g(\emptyslot, \phi \emptyslot) \) for some \( (1, 1) \)-tensor field \( \phi \) satisfying \( \phi^{2} = -\identityOperator + \alpha\tensorproduct\Reebfield \).

Note that although the vector fields \( v_{\pm} \) might not be defined globally, the functions \( p \), \( q \), and \( r \) are.
Indeed, since there locally only are two choices for the pair \( (v_{+}, v_{-}) \), the other being \( (-v_{+}, -v_{-}) \) (see the discussion in the proof of \cref{thm:vpm-Lie-brackets}), the \( 2 \)-tensor fields \( v_{+}^{\flat} \tensorproduct v_{+}^{\flat} \) (and so on) are independent of the choice of local frame.

One then calculates using \cref{thm:vpm-Lie-brackets,thm:LRphi-properties} that
\begin{equation*}
	\Liederivative{\Reebfield} v_{\pm}^{\flat}
	=
	(\Liederivative{\Reebfield} v_{\pm})^{\flat}
	+
	(\Liederivative{\Reebfield}g)(v_{\pm}, \emptyslot)
	=
	\pm \mu v_{\pm}^{\flat}
	+
	2g(h\phi v_{\pm}, \emptyslot)
	=
	\mp \mu v_{\pm}^{\flat}
	.
\end{equation*}
We can compactly write the matrices representing respectively \( \tilde g \) and \( \Liederivative{\Reebfield}{\tilde{g}} \) in the \( (\Reebfield, v_{-}, v_{+}) \) frame as
\begin{equation}
	[\tilde g]
	=
	\begin{pmatrix}
		1&	0&	0	\\
		0&	p&	r	\\
		0&	r&	q	\\
	\end{pmatrix}
	,\quad
	[\Liederivative{\Reebfield}{\tilde{g}}]
	=
	\begin{pmatrix}
		0&	0&	0\\
		0
		& \Reebfield(p) + 2\mu p
		& \Reebfield(r)
		\\
		0
		& \Reebfield(r)
		& \Reebfield(q) - 2\mu p
	\end{pmatrix}
	\label{eq:compatible-metric-matrix-representation-given-critical}
	,
\end{equation}
where \( \Reebfield(p) \) denotes the derivative of the function \( p \) in the direction of the Reeb vector field.
The energy is calculated as
\begin{equation*}
	\norm{\Liederivative{\Reebfield}\tilde g}^2
	=
	\trace\bigl(
		(
			[\tilde g]^{-1}
			[\Liederivative{\Reebfield}{\tilde{g}}]
		)^{2}
	\bigr)
	.
\end{equation*}
Using that \( pq-r^{2} = 1 \) and its derivative \( \Reebfield(p)q + \Reebfield(q)p = 2r\Reebfield(r) \) one can compute (similarly as to how it can be done for contact manifolds~\cite{mitsumatsuExistenceCriticalCompatible2025}) this trace as
\begin{equation}
	\norm{\Liederivative{\Reebfield}\tilde g}^2
	=
	8(1+r^{2}) \mu^{2}
	+
	4(
		q\Reebfield(p) - p\Reebfield(q)
	) \mu
	+
	2(
		\Reebfield(r)^{2} - \Reebfield(p)\Reebfield(q)
	)
	\label{eq:torsion-via-critical-metric-first-expansion}
	.
\end{equation}
The function \( p \) is positive, so one can replace \( q \) by \( p^{-1}(r^{2} + 1) \) and \( \Reebfield(q) \) by \( p^{-1}(\Reebfield(r^{2}) - (r^{2} + 1)\Reebfield(\ln p)) \).
The coefficient to \( \mu \) in \cref{eq:torsion-via-critical-metric-first-expansion} becomes
\begin{equation*}
	4\big(
		2(r^{2}+1)\Reebfield(\ln p)
		-
		2r\Reebfield(r)
	\big)
	=
	8r\big(
		r\Reebfield(\ln p) - \Reebfield(r)
	\big)
	+
	8\Reebfield(\ln p)
	,
\end{equation*}
while the term with no \( \mu \)-coefficients turns into
\begin{equation*}
	2\big(
		\Reebfield(r)^{2}
		-
		2r\Reebfield(r)\Reebfield(\ln p)
		+
		(r^{2}+1)\Reebfield(\ln p)^{2}
	\big)
	\\
	=
	2\big(
		\Reebfield(r) - r\Reebfield(\ln p))
	\big)^{2}
	+
	2\Reebfield(\ln p)^2
	.
\end{equation*}
Collecting the terms gives
\begin{equation*}
	\norm{\Liederivative{\Reebfield}\tilde g}^2
	=
	8\mu^2
	+
	2
	\Bigl(
		(2\mu r)
		+
		(r \Reebfield(\ln p) - \Reebfield(r))
	\Bigr)^2
	+
	2(\Reebfield(\ln p ))^2
	+
	8\mu \Reebfield(\ln p)
	.
\end{equation*}

Recall that \( \norm{\Liederivative{\Reebfield}g}^2 = 8\mu^{2} \).
Since the Reeb vector field is divergence free and \( \mu \) is constant (in fact, \( \Reebfield(\mu) = 0 \), which follows from \cref{thm:scalar-torsion-is-first-integral-for-Reeb-field}, would suffice), the integral of the term \( 8\mu \Reebfield(\ln p) \) vanishes, which lets us conclude that
\begin{equation*}
	E(\tilde g)
	-
	E(g)
	=
	\int_{M}
		\Big[2
		\Bigl(
			(2\mu r)
			+
			(r \Reebfield(\ln p) - \Reebfield(r))
		\Bigr)^2
		+
		2(\Reebfield(\ln p ))^2
	\Big ]\,\alpha\wedge\beta
	\geq
	0
	.
\end{equation*}
Since \( \tilde{g} \) was an arbitrary compatible metric this shows that the critical metric \( g \) minimizes the energy among all compatible metrics.

\printbibliography

\cleardoublepage

\appendix

\markboth{APPENDIX}{APPENDIX}

\section[Appendix]{Calculations for the critical point condition of the Chern--Hamilton energy functional}
\label{sec:critical-point-condition-calculations}

In this Section we carry out the calculations needed to characterize the critical compatible metrics for the Chern--Hamilton energy functional, as claimed in \cref{thm:Euler--Lagrange}.
This will be done by generalizing the calculations by Tanno~\cite{tannoVariationalProblemsContact1989} (who proved the equivalent of \cref{thm:Euler--Lagrange} for contact manifolds) to \( \Reebfield \)-invariant almost cosymplectic manifolds (\cref{def:R-invariant-almost-cosymplectic}) which include both contact and cosymplectic manifolds.
In particular, all the following results hold true in both cases.
If one assumes that the almost cosymplectic structure is contact, the calculations presented here recover the work of Tanno.

We will first show various properties of compatible metrics.
These will then be used to calculate the first variation of the Chern--Hamilton energy functional in \cref{sec:first-variation-calculation} and to compute the tangent space to the space of compatible metrics in \cref{sec:compatible-metrics-tangent-space}.

Let \( (M, \alpha, \beta) \) be an almost cosymplectic manifold of dimension~$2n+1\geq 3$ with Reeb vector field \( \Reebfield \).
If \( g_{t} \) is a curve of compatible metrics we write \( \phi_{t} \) for the associated curve of \( (1, 1) \)-tensor fields, that is,
\begin{equation*}
	\phi_{t}^2 = -\identityOperator + \alpha \tensorproduct \Reebfield
	,\quad\text{and}\quad
	\beta = g_{t}(\emptyslot, \phi_{t}\emptyslot)
	.
\end{equation*}
Let \( \nabla^t \) denote the Levi--Civita connection coming from the metric \( g_t \).
We will be doing computations in index notation and using Einstein summation convention, for instance
\begin{equation*}
	\norm{\Liederivative{\Reebfield}g_t}^2
	=
	g_t^{ir}g_t^{js}
	(\Liederivative{\Reebfield} g_t)_{ij}
	(\Liederivative{\Reebfield} g_t)_{rs}
	.
\end{equation*}

\subsection{Properties satisfied by compatible metrics}
\label{sec:compatible-metrics-properties}

In this Section we will prove some identities which are satisfied by metrics compatible with an almost cosymplectic structure.

\begin{lemma}
	\label{thm:compatible-metric-properties}
	Let \( g \) be a compatible metric on an almost cosymplectic manifold \( (M, \alpha, \beta) \) and \( \phi \) an associated \( (1, 1) \)-tensor field.
	Then,
	\begin{enumerate}
		\item If \( \Liederivative{\Reebfield} \alpha = 0 \), the integral curves of the Reeb field are geodesics: \( \nabla_\Reebfield \Reebfield = 0 \).
			\label{thm:Reebfield-parallel}

		\item \( \Liederivative{\Reebfield}\alpha = 0 \) if and only if \( \phi (\Liederivative{\Reebfield}\phi) + (\Liederivative{\Reebfield}\phi) \phi = 0 \).
			\label{thm:phi-L_Rphi-anticommutes}

		\item Assume that \( \Liederivative{\Reebfield}\alpha = 0 \) and \( \Liederivative{\Reebfield}\beta = 0 \).
			Then, \( (\Liederivative{R}g)(\phi \emptyslot, \emptyslot) = (\Liederivative{R}g)(\emptyslot, \phi \emptyslot) \).
			\label{thm:phi-is-L_Rg-symmetric}

		\item Assume that \( \Liederivative{\Reebfield}\alpha = 0 \), that \( \Liederivative{\Reebfield}\beta = 0 \), and that \( (\dif \alpha)(\phi \emptyslot, \emptyslot) + (\dif \alpha)(\emptyslot, \phi \emptyslot) = 0 \).
			Then both \( \nabla_{\phi\emptyslot}\alpha \) and \( (\nabla \alpha) \phi \) are symmetric \( 2 \)-tensor fields.
			\label{thm:nablaalphaphiCommutator}

		\item Assume that \( M \) is of dimension~\( 3 \) or that \( (\nabla \alpha)\phi \) is symmetric and that \( \Liederivative{\Reebfield}\alpha = 0 \) and \( \Liederivative{\Reebfield}\beta = 0 \).
			Then \( \nabla_\Reebfield \beta = 0 \), and equivalently \( \nabla_\Reebfield \phi = 0 \).
			\label{thm:nabla_Rbeta=0}
			\label{thm:nabla_Rphi=0}
	\end{enumerate}
	\begin{enumerate}[resume]
		\item If \( g_{t} \) is a curve of compatible metrics, then \( g'_{t} = -g'_{t}(\phi_t \emptyslot, \phi_t \emptyslot) \).
			\label{thm:2hphiIdentity}
	\end{enumerate}
\end{lemma}

\begin{remark}
	All the extra assumptions hold for \( \Reebfield \)-invariant almost cosymplectic structures and thus for both cosymplectic and contact manifolds.
	The above identities are thus a generalization of what is already known for contact manifolds, see for example~\cite[Lemma 1.1 and Equation (4.9)]{tannoVariationalProblemsContact1989}.
\end{remark}

\begin{proof}[Proof of \cref{thm:compatible-metric-properties}]
	In \cref{thm:Reebfield-parallel} we can write \( 0 = \Liederivative{\Reebfield}\alpha = \nabla_{\Reebfield}\alpha + \alpha(\nabla \Reebfield) \).
	But \( 
		2\alpha(\nabla \Reebfield)
		=
		2g(\Reebfield, \nabla \Reebfield)
		=
		\nabla \norm{\Reebfield}^2
		=
		0
		,
	\)
	since \( \norm{\Reebfield}^{2} = \alpha(\Reebfield) = 1 \).
	Thus \( \nabla_\Reebfield \alpha \) vanishes which implies that \( \nabla_\Reebfield \Reebfield \) does as well since \( \Reebfield \) is the metric dual of \( \alpha \).

	For \cref{thm:phi-L_Rphi-anticommutes}, calculate that
	\begin{equation*}
		(\Liederivative{\Reebfield}\phi) \phi
		+
		\phi (\Liederivative{\Reebfield}\phi)
		=
		\Liederivative{\Reebfield}(\phi^2)
		=
		\Liederivative{\Reebfield}(-1 + \alpha \tensorproduct \Reebfield)
		=
		(\Liederivative{\Reebfield}\alpha) \tensorproduct \Reebfield
		.
	\end{equation*}

	For \cref{thm:phi-is-L_Rg-symmetric}, combine the fact that \( \Liederivative{\Reebfield}g \) is symmetric with the assumptions that \( \Liederivative{\Reebfield}\alpha = 0 \) and \( \Liederivative{\Reebfield} \beta = 0 \) to get that
	\begin{equation*}
	\begin{split}
		\Liederivative{\Reebfield}g
		&=
		\Liederivative{\Reebfield}(
			g(\phi \emptyslot, \phi \emptyslot)
			+
			\alpha \tensorproduct \alpha
		)
		=
		\Liederivative{\Reebfield}
		\bigl(
			\beta(\phi \emptyslot, \emptyslot)
		\bigr)
		=
		(\Liederivative{\Reebfield}\beta)(\phi \emptyslot, \emptyslot)
		+
		\beta((\Liederivative{\Reebfield} \phi) \emptyslot, \emptyslot)
		\\
		&=
		g(
			(\Liederivative{\Reebfield}\phi) \emptyslot,
			\phi \emptyslot
		)
		.
	\end{split}
	\end{equation*}
	By \cref{thm:phi-L_Rphi-anticommutes}, \( \phi \) anticommutes with \( \Liederivative{\Reebfield}\phi \), so this and the antisymmetry of \( \phi \) imply the claim.

	For \cref{thm:nablaalphaphiCommutator}, the assumptions mean that \cref{thm:phi-is-L_Rg-symmetric} holds, so
	\begin{equation*}
	\begin{split}
		0
		&=
		\bigl(
			(\dif\alpha)(\phi X, Y)
			+
			(\dif\alpha)(X, \phi Y)
		\bigr)
		-
		\bigl(
			(\Liederivative{R}g)(\phi X, Y)
			-
			(\Liederivative{R}g)(X, \phi Y)
		\bigr)
		\\
		&=
		\bigl(
			(\nabla_{\phi X}\alpha)Y
			-
			(\nabla_{Y}\alpha) \phi X
		\bigr)
		+
		\bigl(
			(\nabla_{X}\alpha)\phi Y
			-
			(\nabla_{\phi Y}\alpha) X
		\bigr)
		\\
		&-
		\bigl(
			(\nabla_{\phi X}\alpha)Y + (\nabla_{Y}\alpha) \phi X
		\bigr)
		+
		\bigl(
			(\nabla_{X}\alpha)\phi Y + (\nabla_{\phi Y}\alpha) X
		\bigr)
		\\
		&=
		2
		\bigl(
			(\nabla_{X} \alpha) \phi Y
			-
			(\nabla_{Y} \alpha) \phi X
		\bigr)
		.
	\end{split}
	\end{equation*}
	Adding the \( \Liederivative{\Reebfield}g \) terms to the \( \dif \alpha \) ones instead yields the other claim.

	For \cref{thm:nabla_Rbeta=0}, the equivalence of the two statements is simple to see: \( \nabla g = 0 \), so
	\begin{equation*}
		\nabla_{\Reebfield}\beta
		=
		\nabla_{\Reebfield} ( g(\emptyslot, \phi \emptyslot ))
		=
		(\nabla_{\Reebfield}g)(\emptyslot, \phi \emptyslot)
		+
		g(\emptyslot, (\nabla_\Reebfield\phi)\emptyslot)
		=
		g(\emptyslot, (\nabla_\Reebfield\phi)\emptyslot)
		.
	\end{equation*}
	This is essentially the fact that raising and lowering indices commutes with covariant derivatives.
	To prove that \( \nabla_{\Reebfield}\beta = 0 \) under the assumption that \( M \) has dimension~\( 3 \), pick some local orthonormal frame \( (X, \phi X, R) \).
	It is enough to verify that \( \nabla_\Reebfield\beta \) vanishes on this frame.

	Since \( \nabla_{\Reebfield} \Reebfield = 0 \) by \cref{thm:Reebfield-parallel} and \( \iota_\Reebfield\beta = 0 \), also \( \iota_\Reebfield \nabla_{\Reebfield} \beta = 0 \).
	Moreover, \( \nabla_\Reebfield \beta \) is antisymmetric, so we now only need to calculate \( (\nabla_\Reebfield\beta)(X, \phi X) \):
	\begin{multline*}
		(\nabla_\Reebfield\beta)(X, \phi X)
		=
		\Reebfield(g(X, \phi^2 X))
		-
		g(\nabla_\Reebfield X, \phi^2 X)
		-
		g(X, \phi \nabla_\Reebfield (\phi X))
		\\
		=
		-\Reebfield(g(X, X))
		+
		g(\nabla_\Reebfield X, X)
		+
		g(\phi X, \nabla_\Reebfield (\phi X))
		=
		0
		,
	\end{multline*}
	since \( 0 = R(g(X, X)) = 2g(\nabla_\Reebfield X, X) \), (and the same for \( \phi X \)) since we chose a normalized frame.

	Now we prove it instead in arbitrary dimension with the assumption of symmetry of \( (\nabla \alpha)\phi \).
	We first use that \( 0 = \Liederivative{\Reebfield}\beta = \nabla_\Reebfield\beta + \beta(\nabla_\emptyslot \Reebfield, \emptyslot) + \beta(\emptyslot, \nabla_\emptyslot \Reebfield) \).
	Then we can rewrite:
	\begin{equation*}
		\beta(\nabla_X \Reebfield, Y)
		=
		g(\nabla_X \Reebfield, \phi Y)
		=
		Xg(\Reebfield, \phi Y)
		-
		g(\Reebfield, \nabla_X \phi Y)
		=
		-\alpha(\nabla_X \phi Y)
		.
	\end{equation*}
	We can calculate that
	\( 
		(\nabla_\emptyslot\alpha)(\phi \emptyslot)
		=
		\Liederivative{\emptyslot}(\alpha \circ \phi(\emptyslot))
		-
		\alpha(\nabla_\emptyslot(\phi \emptyslot))
	\).
	Since \( \alpha \circ \phi = 0 \), \cref{thm:nablaalphaphiCommutator} implies that \( \alpha(\nabla_\emptyslot (\phi \emptyslot)) \) is symmetric.
	This shows that \( \beta(\nabla_\emptyslot \Reebfield, \emptyslot) \) is symmetric, so since \( \beta \) is antisymmetric we finally conclude that \( \nabla_\Reebfield \beta = \Liederivative{\Reebfield} \beta = 0 \).

	The identity in \cref{thm:2hphiIdentity} comes from taking the time derivative of \( g_{t} = g_{t}(\phi_{t} \emptyslot, \phi_{t} \emptyslot) + \alpha\tensorproduct\alpha \)
	and using that \( \beta \) does not depend on time which implies that
	\(
		0
		=
		\beta'
		=
		g_{t}'(\emptyslot, \phi_{t} \emptyslot)
		+
		g_{t}(\emptyslot, \phi_{t}' \emptyslot)
	\).
	In particular, \( g_{t}'(\phi_{t} \emptyslot, \phi_{t} \emptyslot) = - g_{t}(\phi_{t} \emptyslot, \phi_{t}' \emptyslot) \), so
	\begin{equation*}
		\odfrac{}{t} g_{t}(\phi \emptyslot, \phi \emptyslot)
		=
		g_{t}'(\phi_{t} \emptyslot, \phi_{t} \emptyslot)
		+
		g_{t}(\phi_{t}' \emptyslot, \phi_{t} \emptyslot)
		+
		g_{t}(\phi_{t} \emptyslot, \phi_{t}'\emptyslot)
		-g_{t}'(\phi_{t}\emptyslot, \phi_{t}\emptyslot)
		.
	\end{equation*}
\end{proof}

\subsection{Calculating the first variation of the Chern--Hamilton energy functional}
\label{sec:first-variation-calculation}
Let \( (M, \alpha, \beta) \) be an \( \Reebfield \)-invariant almost cosymplectic manifold and let \( g_{t} \) be a curve of compatible metrics with \( \phi_t \) the \( (1, 1) \)-tensor fields associated to \( g_t \).
We now calculate the derivative of the Chern--Hamilton energy along this curve:
\begin{equation*}
	\odfrac{}{t}\
	\int
		\norm{\Liederivative{\Reebfield} g_{t}}^2
	\,
	\alpha\wedge\beta^n
	=
	2
	\int
	\Big(g_{t}(
		\Liederivative{\Reebfield} g_{t},
		\Liederivative{\Reebfield} g'_{t}
	)
	-
	g_{t}^{ir} (g'_{t})^{js} (\Liederivative{\Reebfield} g_{t})_{ij} (\Liederivative{\Reebfield} g_{t})_{rs} 
	\Big)\,
	\alpha\wedge\beta^n
	.
\end{equation*}
We show below (\cref{thm:energy-variation-cross-terms}) that the second term vanishes.
The first term becomes
\begin{equation*}
\begin{split}
	g_{t}\bigl(
		\Liederivative{\Reebfield} g_{t},
		\Liederivative{\Reebfield} g_{t}'
	\bigr)
	&=
	g_{t}\bigl(
		\Liederivative{\Reebfield} g_{t},
		\nabla^{t}_\Reebfield g_{t}' + g_{t}'(\nabla^{t}_{\emptyslot} \Reebfield, \emptyslot) + g_{t}'(\emptyslot, \nabla^{t}_{\emptyslot} \Reebfield)
	\bigr)
	\\
	&=
	\Reebfield\left( 
		g_{t}(\Liederivative{\Reebfield} g_{t}, g_{t}')
	\right)
	\\
	&-
	g_{t}(\nabla^{t}_\Reebfield \Liederivative{\Reebfield} g_{t}, g_{t}')
	+
	g_{t}\Bigl(
		\Liederivative{\Reebfield} g_{t},
		g_{t}'(\nabla^{t}_{\emptyslot} \Reebfield, \emptyslot)
		+
		g_{t}'(\emptyslot, \nabla^{t}_{\emptyslot} \Reebfield)
	\Bigr)
	.
\end{split}
\end{equation*}
One obtains the first equality by the definition of the Lie derivative of tensor fields combined with the fact that the Levi--Civita connection is torsion-free.
The identity \( X(g(S, T)) = g(\nabla_X S, T) + g(S, \nabla_X T) \) holds for vector fields \( X \) and tensor fields \( S \) and \( T \).
The identity holds when \( S \) and \( T \) are vector fields or \( 1 \)-tensor fields by definition, and for general tensor fields because \( \nabla_X \) satisfies a product rule with regards to tensor products, and because the inner product of a tensor product is the product of the inner products.
The second line is a divergence (because of the identity \( \divergence (fX) = f\divergence X + \dif f(X) \) and that \( \Reebfield \) is divergence free) and vanishes when taking the integral.

In \cref{thm:energy-variation-cross-terms} we compute that the last term equals
\(
	2
	g_{t}\bigl(
		(\Liederivative{R}g)
		(
			\emptyslot,
			\difalphasharp \emptyslot
		)
		,
		g'_{t}
	\bigr)
\) where \( \difalphasharp\) is defined by \( g(\emptyslot, \difalphasharp \emptyslot) = \dif\alpha \), in coordinates \( \tensor{(\difalphasharp)}{^i_j} := \tensor{\dif\alpha}{^i_j} \).
In conclusion we get that
\begin{equation*}
	\odfrac{}{t}\
	\int
		\norm{\Liederivative{\Reebfield} g_{t}}^2
	\,
	\alpha\wedge\beta^n
	=
	2\int_{M}\Big(
	-g_{t}\left( 
		\nabla^t_\Reebfield \Liederivative{\Reebfield} g_{t},
		g_{t}'
	\right)
	+
	2
	g_{t}\Bigl(
		(\Liederivative{R}g)
		\bigl(
			\emptyslot,
			\difalphasharp \emptyslot
		\bigr)
		,
		g'_{t}
	\Bigr)\Big)
	\,\alpha\wedge\beta^n
	.
\end{equation*}

What is missing now is calculating the two terms mentioned above.

\begin{lemma}
	Let \( (M, \alpha, \beta) \) be an \( \Reebfield \)-invariant almost cosymplectic manifold and let \( g_{t} \) be a path of compatible metrics with associated curve of \( (1, 1) \)-tensor fields \( \phi_{t} \).
	Then,
	\begin{align}
		g_{t}^{ir}(g'_{t})^{js}
		(\Liederivative{\Reebfield} g_{t})_{ij}
		(\Liederivative{\Reebfield} g_{t})_{rs}
		&=
		0
		,
		\label{eq:energyCrossterm1}
	\intertext{and}
	g_{t}(
		\Liederivative{\Reebfield} g_{t},
		g'_{t}(\nabla^{t} \Reebfield, \emptyslot)
		+
		g'_{t}(\emptyslot, \nabla^{t} \Reebfield)
	)
	&=
	2
	g_{t} \Bigl(
		(\Liederivative{\Reebfield} g)
		\bigl(
			\emptyslot,
			\difalphasharp \emptyslot
		\bigr)
		,
		g'_{t}
	\Bigr)
	.
		\label{eq:energyCrossterm2}
	\end{align}
	\label{thm:energy-variation-cross-terms}
\end{lemma}

\begin{proof}
	Let us start with \cref{eq:energyCrossterm1}.
	We will need that
	\begin{equation}
		\iota_{\Reebfield} \Liederivative{\Reebfield} g
		=
		R(\alpha(\emptyslot))
		-
		g(\Liederivative{\Reebfield} \Reebfield, \emptyslot)
		-
		\alpha(\Liederivative{\Reebfield} \emptyslot)
		=
		\Liederivative{\Reebfield} \alpha
		=
		0
		.
		\label{eq:iRLRg=0}
	\end{equation}
	Now, using \cref{thm:compatible-metric-properties}, \cref{thm:2hphiIdentity} (adding \( g'_{t} \) on both sides of that identity to get \( 2g'_{t} = g'_{t} - g'_{t}(\phi_{t} \emptyslot, \phi_{t} \emptyslot) \)):
	\begin{multline}
	\label{eq:firstVariationFirstCrossTerm}
		2
		g_{t}^{ir}
		(g'_{t})^{js}
		(\Liederivative{\Reebfield} g_{t})_{ij}
		(\Liederivative{\Reebfield} g_{t})_{rs}
		\\
		=
		g_{t}^{ir}
		\bigl(
			(g'_{t})^{js}
			-
			(g'_{t})^{ab} \tensor{(\phi_t)}{^j_a} \tensor{(\phi_t)}{^s_b}
		\bigr)
		\bigl(
			(\nabla^{t}_i\alpha)_j
			+
			(\nabla^{t}_j\alpha)_i
		\bigr)
		\bigl(
			(\nabla^{t}_r\alpha)_s
			+
			(\nabla^{t}_s\alpha)_r
		\bigr)
		.
	\end{multline}
	Now let us reduce the term
	\begin{equation*}
		-g_{t}^{ir} (g'_{t})^{ab}
		\tensor{(\phi_t)}{^j_a}
		\tensor{(\phi_t)}{^s_b}
		\bigl(
			(\nabla^{t}_i\alpha)_j
			+
			(\nabla^{t}_j\alpha)_i
		\bigr)
		\bigl(
			(\nabla^{t}_r\alpha)_s
			+
			(\nabla^{t}_s\alpha)_r
		\bigr)
		.
	\end{equation*}
	Applying \cref{thm:nablaalphaphiCommutator} of \cref{thm:compatible-metric-properties} to all the four occurrences of \( \nabla\alpha \) composed with \( \phi_t \) gives
	\begin{equation*}
		-g_t^{ir}
		(g'_{t})^{ab}
		\tensor{(\phi_t)}{^j_i}
		\tensor{(\phi_t)}{^s_r}
		\bigl(
			(\nabla^{t}_a\alpha)_j
			+
			(\nabla^{t}_j\alpha)_a
		\bigr)
		\bigl(
			(\nabla^{t}_b\alpha)_s
			+
			(\nabla^{t}_s\alpha)_b
		\bigr)
		.
	\end{equation*}
	Then raise the \( i \) indices using \( g_t^{ir} \), move the upper \( j \)-indices down and vice versa, and use that \( \phi \) is antisymmetric to replace the resulting \( \tensor{\phi}{_j^r} \) by \( -\tensor{\phi}{^r_j} \) to get
	\begin{equation*}
		\begin{split}
		&\phantom{=-}(g'_{t})^{ab}
		\tensor{(\phi_t)}{^r_j}
		\tensor{(\phi_t)}{^s_r}
		\bigl(
			(\nabla^{t}_a\Reebfield)^j
			+
			( (\nabla^{t})^j\alpha)_a
		\bigr)
		(\Liederivative{\Reebfield}g)_{bs}
		\\
		&=
		-
		(g'_{t})^{ab}
		(
			\delta^s_j
			-
			\alpha_j
			\Reebfield^s
		)
		\bigl(
			(\nabla^{t}_a\Reebfield)^j
			+
			( (\nabla^{t})^j\alpha)_a
		\bigr)
		(\Liederivative{\Reebfield}g)_{bs}
		\\
		&=
		-
		(g'_{t})^{ab}
		(
			g^{is}
			\delta_i^j
			-
			\Reebfield^j
			\Reebfield^s
		)
		(\Liederivative{\Reebfield}g)_{aj}
		(\Liederivative{\Reebfield}g)_{bs}
		\\
		&=
		-
		g^{is} 
		(g'_{t})^{ab}
		(\Liederivative{\Reebfield}g)_{ai}
		(\Liederivative{\Reebfield}g)_{bs}
		,
	\end{split}
	\end{equation*}
	since \( \contractwith{\Reebfield} \Liederivative{\Reebfield}g = 0 \).
	This is exactly the negative of the other term in \cref{eq:firstVariationFirstCrossTerm} (after renaming indices), so we are done with proving \cref{eq:energyCrossterm1}.
	Note that if \( \alpha \) is closed, \( (\nabla_i\alpha)_j = (\nabla_j\alpha)_i \), which reduces the size of these expressions (but does not change the result).

	The right hand side of the other term (\cref{eq:energyCrossterm2}) can be slightly simplified:
	\begin{equation*}
		(\Liederivative{\Reebfield} g_{t})^{ij}
		\left( 
			(g'_{t})_{rj} (\nabla^{t}_i \Reebfield)^r
			+
			(g'_{t})_{ir} (\nabla^{t}_j \Reebfield)^r
		\right)
		=
		2
		(\Liederivative{\Reebfield} g_{t})^{ij}
		(g'_{t})_{ir} (\nabla^{t}_j \Reebfield)^r
		,
	\end{equation*}
	since both \( \Liederivative{\Reebfield} g_{t} \) and \( g'_{t} \) are symmetric.
	Using \cref{thm:compatible-metric-properties}, \cref{thm:2hphiIdentity}, as before, this can be written as
	\begin{equation}
		\label{eq:first-variation-second-cross-term-one-side}
		2
		(\Liederivative{\Reebfield} g_{t})^{ij}
		(g'_{t})_{ir} (\nabla^{t}_j \Reebfield)^r
		=
		(\Liederivative{\Reebfield} g_{t})^{ij}
		\bigl(
			(g'_{t})_{ir}
			-
			(g'_{t})_{ab}
			\tensor{(\phi_t)}{^a_i}
			\tensor{(\phi_t)}{^b_r}
		\bigr)
		(\nabla^{t}_j \Reebfield)^r
		.
	\end{equation}
	Now work on the second term.
	We use a slight variant of \cref{thm:nablaalphaphiCommutator} from \cref{thm:compatible-metric-properties}, namely that
	\( 
		\tensor{(\phi_t)}{^b_r}
		(\nabla^{t}_a \Reebfield)^r
		=
		\tensor{(\phi_t)}{^r_a}
		((\nabla^{t})^b \alpha)_r
	\),
	as well as \cref{thm:phi-is-L_Rg-symmetric} from the same \namecref{thm:compatible-metric-properties} to calculate
	\begin{equation*}
	\begin{split}
		-(g'_{t})_{ab}
		&\tensor{(\phi_t)}{^a_i}
		\tensor{(\phi_t)}{^b_r}
		(\nabla^{t}_j\Reebfield)^r
		(\Liederivative{\Reebfield} g_{t})^{ij}
		\\
		&=
		(g'_{t})_{ab}
		\tensor{(\phi_t)}{^j_i}
		\tensor{(\phi_t)}{^r_j}
		((\nabla^{t})^{b}\alpha)_rj
		(\Liederivative{\Reebfield} g_{t})^{ia}
		\\
		&=
		-(g'_{t})_{ab}
		(
			\delta^r_i
			-
			\alpha_i \Reebfield^r
		)
		((\nabla^{t})^{b}\alpha)_{rj}
		(\Liederivative{\Reebfield} g_{t})^{ia}
		\\
		&=
		-(g_t')_{ab}
		((\nabla^{t})^b\alpha)_i
		(\Liederivative{\Reebfield} g_{t})^{ia}
		,
	\end{split}
	\end{equation*}
	where again we use that \( \contractwith{\Reebfield} \Liederivative{\Reebfield}g_t = 0 \).
	Relabelling \( (a, b, i) \) to \( (i, r, j) \) and combining with the left term of \cref{eq:first-variation-second-cross-term-one-side}, the right hand side of \cref{eq:energyCrossterm2} becomes
	\begin{equation*}
		2(\Liederivative{\Reebfield}g)^{ij}
		(g'_{t})_{ir}
		\bigl(
			(\nabla^{t}_{j} \Reebfield)^r
			-
			((\nabla^{t})^r \alpha)_j
		\bigr)
		=
		2(\Liederivative{\Reebfield}g)^{ij}
		(g'_{t})_{ir}
		\tensor{\dif\alpha}{_j^r}
		,
	\end{equation*}
	which is just
	\(
		g_t\bigl(
			g'_{t}
			,
			(\Liederivative{\Reebfield}g)
			(
				\emptyslot,
				\difalphasharp \emptyslot
			)
		\bigr)
	\)
	as claimed.
\end{proof}

\subsection{The tangent space to the space of compatible metrics}
\label{sec:compatible-metrics-tangent-space}

In this Section we prove \cref{thm:compatible-metrics-tangent-space}, giving a description of the \( (2, 0) \)-tensor fields that lie tangent to curves of compatible metrics.
We also show that the tensor field
\begin{equation}
	H
	:=
	\nabla_\Reebfield \Liederivative \Reebfield g
	-
	2 (\Liederivative{\Reebfield} g)(
		\emptyslot,
		\difalphasharp \emptyslot
	)
	,
	\label{eq:dalphasarp-defined-compatible-metric-tangent}
\end{equation}
where \( g \) is a compatible metric, lies in this tangent space at \( g \).

Recall that the claim was that a \( (2, 0) \)-tensor field \( H \) is tangent to a curve \( g(t) \) of compatible metrics at \( t=0 \) if and only if it is symmetric, has \( \iota_{\Reebfield} H = 0 \), and has \( H(\phi \emptyslot, \emptyslot) = H(\emptyslot, \phi \emptyslot) \), where \( \phi \) is the \( (1, 1) \)-tensor field associated to \( g(0) \).

The proof of the claim for a general \( \Reebfield \)-invariant almost cosymplectic manifold follows the same steps as the proof for the contact case which was done by Tanno~\cite{tannoVariationalProblemsContact1989}.
We provide the generalization here for completeness, and in the notation used all along this manuscript.

\begin{proof}[Proof of \cref{thm:compatible-metrics-tangent-space}]
	First, let \( g(t) \) be a curve of compatible metrics with associated curve of \( (1, 1) \)-tensor fields \( \phi(t) \).
	Since all the \( g(t) \) are symmetric, also the \( g'(t) \) are, and \( \iota_\Reebfield g'(t) = 0 \), since
	\begin{equation*}
		\iota_\Reebfield g'(t)
		=
		\iota_{\Reebfield}
		\odfrac{}{t}
		g(t)(\emptyslot, \emptyslot)
		=
		\odfrac{}{t}
		\alpha
		=
		0
		.
	\end{equation*}
	Now let us verify that \( \phi \) is symmetric for \( g'(t) \).
	By \cref{thm:compatible-metric-properties}, \( g' = -g'(\phi \emptyslot, \phi \emptyslot) \), which implies that \( \phi \) is symmetric for \( g' \), since \( g' \) vanishes on \( \Reebfield \).

	Now for the other direction let \( g_0 \) be a compatible metric with associated \( (1, 1) \)-tensor field \( \phi_0 \) and let a symmetric tensor field \( H \) satisfying the two criteria be given.
	We will now construct a curve \( g(t) \) of compatible metrics with \( g'(0) = H \) and \( g(0) = g_0 \).
	Start by defining the \( (1, 1) \)-tensor field \( H^+ \) by \( g_{0}(\emptyslot, H^+ \emptyslot) = H \), and then let
	\begin{equation*}
		g(t) := g_0(e^{t H^+}\emptyslot, \emptyslot),
		\quad\text{and}\quad
		\phi(t) := \phi_0 e^{tH^+}
		.
	\end{equation*}
	There are no convergence problems with the operators \( e^{tH^+} \) since they are essentially defined by a matrix exponential at each tangent space of \( M \).
	Since \( H \) and thus \( H^+ \) are symmetric, we can write \( g(t) = g_0(e^{\frac{1}{2} tH^+}\emptyslot, e^{\frac{1}{2} tH^{+}} \emptyslot) \) to see that \( g(t) \) is positive and symmetric and thus actually a metric.
	Let us verify that it indeed defines a compatible metric.
	For this we need to show that \( \phi(t)^2 = -\identityOperator + \alpha \tensorproduct \Reebfield \) and that \( g(t)(\phi(t) \emptyslot, \phi(t) \emptyslot) + \alpha \tensorproduct \alpha = g(t) \).
	Since \( \phi_0 \) is symmetric for \( H \) and antisymmetric for \( g_0 \) it anticommutes with \( H^+ \).
	This implies that \( e^{tH^+}\phi_0 = \phi_0 e^{-tH^+} \).
	In particular, \( \phi(t)^2 = \phi_0^2 = -\identityOperator + \alpha \tensorproduct \Reebfield \).
	Combining the (anti)symmetries, we also see that
	\begin{equation*}
	\begin{split}
		g(t)(\phi(t) \emptyslot, \phi(t) \emptyslot)
		&=
		g_0(
			e^{th^+} \phi_0 e^{tH^+} \emptyslot,
			\phi_0 e^{tH^+} \emptyslot
		)
		=
		g_0(
			\phi_0 e^{tH^+} \emptyslot,
			\phi_0 \emptyslot
		)
		\\
		&=
		g(t)
		-
		\bigl(
			\alpha \circ e^{tH^+}
		\bigr)
		\tensorproduct \alpha
		=
		g(t)
		-
		\alpha \tensorproduct \alpha
		,
	\end{split}
	\end{equation*}
	where the last equality uses that \( \exp(tH^+) \Reebfield = \Reebfield \) since \( \contractwith{\Reebfield} H = 0 \).
	We have thus proven that \( g(t) \) is a curve of compatible metrics.
\end{proof}

\begin{lemma}
	Let \( g \) be a compatible metric on an \( \Reebfield \)-invariant almost cosymplectic manifold \( (M, \alpha, \beta) \).
	Then the expression \cref{eq:dalphasarp-defined-compatible-metric-tangent} defines an element in the tangent space to the space of compatible metrics at \( g \).
\end{lemma}

\begin{proof}
	We need to verify that
	\(
		H
		=
		\nabla_\Reebfield \Liederivative{\Reebfield}g
		-
		2(\Liederivative{\Reebfield}g)
		(
			\emptyslot,
			\difalphasharp \emptyslot
		)
	\)
	is symmetric and that both \( \iota_\Reebfield H = 0 \) and \( H(\phi \emptyslot, \emptyslot) = H(\emptyslot, \phi \emptyslot) \).
	By \cref{eq:iRLRg=0}, \( \iota_{\Reebfield} \Liederivative{\Reebfield} g = 0 \), so
	\begin{equation*}
		(\nabla_\Reebfield \Liederivative{\Reebfield} g)(R, \emptyslot)
		=
		\Reebfield(\iota_{\Reebfield} \Liederivative{\Reebfield}g)
		-
		(\Liederivative{\Reebfield}g)(\nabla_\Reebfield \Reebfield, \emptyslot)
		-
		(\Liederivative{\Reebfield}g)(\Reebfield, \nabla_{\Reebfield} \emptyslot)
		=
		0
		.
	\end{equation*}
	The other term of \( H \) also vanishes on \( \Reebfield \) since \( \iota_\Reebfield \dif \alpha = \Liederivative{\Reebfield}\alpha - \dif \iota_{\Reebfield}\alpha = 0 - \dif 1 = 0 \).

	For the \( \phi \)-symmetry recall that \cref{thm:compatible-metric-properties} tells us that \( \phi \) is symmetric for \( \Liederivative{\Reebfield}g \) and commutes with \( \nabla_{\Reebfield} \).
	Writing out
	\begin{equation*}
		(\nabla_{\Reebfield} \Liederivative{\Reebfield}g)(\phi \emptyslot, \emptyslot)
		=
		\Reebfield(
			(\Liederivative{\Reebfield}g)(\phi \emptyslot, \emptyslot)
		)
		-
		(\Liederivative{\Reebfield}{g})(\nabla_\Reebfield(\phi \emptyslot), \emptyslot)
		-
		(\Liederivative{\Reebfield}{g})(
			\phi \emptyslot,
			\nabla_{\Reebfield} \emptyslot
		)
	\end{equation*}
	we then see that indeed \( \phi \) is symmetric for \( \nabla_{\Reebfield} \Liederivative{\Reebfield}g \).
	For the other term, the assumption that \( \phi \) is antisymmetric for \( \dif\alpha \) means that it commutes with \( \difalphasharp \) since \( \phi \) is antisymmetric for \( g_0 \).
	Combining this with \( \phi \) being symmetric for \( \Liederivative{\Reebfield}g \) means that \( \phi \) is also symmetric for \( (\Liederivative{\Reebfield}g)(\emptyslot, \difalphasharp \emptyslot) \).

	For the symmetry of \( H \), its first term is symmetric since \( g \) is.
	For the second term, the Lie derivative of \( g = g(\phi \emptyslot, \phi \emptyslot) + \alpha \otimes \alpha \) implies (as in \cref{thm:LRphi-properties}) that \( \Liederivative{\Reebfield}g = g(\emptyslot, (\Liederivative{\Reebfield}\phi)\phi \emptyslot) \).
	Taking the Lie derivative of \cref{thm:phi-is-L_Rg-symmetric} of \cref{thm:compatible-metric-properties} and using that \( \alpha \) is \( \Reebfield \)-invariant shows that \( \difalphasharp \) and \( \Liederivative{\Reebfield}\phi \) anticommute.
	Combined with the (anti)symmetries and anticommutation properties from \cref{thm:compatible-metric-properties} as well as the commutativity of \( \difalphasharp \) and \( \Liederivative{\Reebfield}\phi \), we see that also the second term of \( H \) is symmetric.
\end{proof}

\end{document}